\documentclass[leqno,final]{siamltex}
\usepackage{amsmath, dsfont}
\usepackage{graphicx,color}
\usepackage{amssymb}

\newtheorem{remark}{Remark}

\newtheorem{assumption}{Assumption}[section]

\def\D{\mathcal{D}}

\def\P{\mathbb{P}}
\def\E{\mathbb{E}}

\def\({\left(}
\def\){\right)}

\begin{document}

\title{Finite Element Approximations of a Class of Nonlinear
Stochastic Wave Equation with Multiplicative Noise}
\markboth{YUKUN LI AND SHUONAN WU AND YULONG XING}{Finite Element Approximations for Stochastic Wave Equation}

\author{
Yukun Li\thanks{Department of Mathematics, University of Central
Florida, Orlando, FL 32816, U.S.A. ({\tt yukun.li@ucf.edu}). The work of this author was partially supported by the NSF grant DMS-2110728.}
\and
Shuonan Wu\thanks{School of Mathematical Sciences, Peking University, 100871, China ({\tt snwu@math.pku.edu.cn}). The work of this author
was partially supported by the National Natural Science Foundation of
China grant No. 11901016.}
\and
Yulong Xing\thanks{Department of Mathematics, The Ohio State
University, Columbus, OH 43210, U.S.A.  ({\tt xing.205@osu.edu}). The work of
this author was partially supported by the NSF grant DMS-1753581.}
}

\maketitle

\begin{abstract}
Wave propagation problems have many applications in physics and
  engineering, and the stochastic effects are important in accurately
  modeling them due to the uncertainty of the media. This paper
  considers and analyzes a fully discrete finite element method for a
  class of nonlinear stochastic wave equations, where the diffusion
  term is globally Lipschitz continuous while the drift term is only
  assumed to satisfy weaker conditions as in
  \cite{chow2002stochastic}. The novelties of this paper are threefold. First, the error estimates cannot not be directly obtained if the numerical scheme in primal form is used. The numerical scheme in mixed form is introduced and several H\"{o}lder continuity results of the strong solution are proved, which are used to establish the error estimates in both $L^2$ norm and energy norms. Second, two types of discretization of the nonlinear term are proposed to establish the $L^2$ stability and energy stability results of the discrete solutions. These two types of discretization and proper test functions are designed to overcome the challenges arising from the stochastic scaling in time issues and the nonlinear interaction. These stability results play key roles in proving the probability of the set on which the error estimates hold approaches one. Third, higher order moment stability results of the discrete solutions are proved based on an energy argument and the underlying energy decaying property of the method. Numerical experiments are also presented to show the stability results of the discrete solutions and the convergence rates in various norms.
\end{abstract}

\begin{keywords}
stochastic wave equation, multiplicative noise, finite element method, higher order moment, error estimate, stability analysis.
\end{keywords}

\begin{AMS}
65N12, 
65N15, 
65N30, 
\end{AMS}

\section{Introduction}\label{sec-1}
In this paper, we consider the nonlinear stochastic wave equation
with Neumann boundary condition and functional-type multiplicative
noise, taking the form of
\begin{align}\label{sac_s}
&du_t = \Delta u dt+f(u)dt+g(u)dW(t),& \mbox{in }
\mathcal{D} \times (0,T],\\
&\frac{\partial u}{\partial n} = 0 & \mbox{in }
\partial\mathcal{D} \times (0,T],\\
&u(0)=h_1(x)\quad u_t(0)=h_2(x). & \mbox{in }\mathcal{D}.\label{sac_s2}
\end{align}
Here $\mathcal{D} \subset \mathbb{R}^d ~(d=1,2,3)$ is a bounded
domain, $W: \Omega \times (0, T] \to \mathbb{R}$ is a standard Weiner
process on the filtered probability space $(\Omega, \mathcal{F},
\{\mathcal{F}_t: t\geq 0\}, \mathbb{P})$.  The nonlinear drift term
$f(u)$ is assumed to satisfy the conditions specified in
\cite{chow2002stochastic}, which discusses the existence and
uniqueness for local and global solutions of \eqref{sac_s} in Sobolev
space. More specifically, we assume
\begin{align}\label{eq20191102_1}
f(u)=\sum_{j=1}^qa_j(x)u^j,
\end{align}
where $q$ is an odd integer with $1\le q\le3$ for $d=3$ and $q\ge1$ for
$d=1,\,2$, $a_j(x)$ are bounded and continuous for any $j$, and there
exist positive constants $\alpha\ge0$ and $\lambda>0$ such that
\begin{align}\label{eq20191102_2}
F(u) := -\int_0^uf(s)ds=-\sum_{j=1}^q\frac{1}{j+1}a_j(x)u^{j+1}\ge\left(\frac{\alpha}{2} + 
\frac{\lambda}{2} u^{q-1}\right)u^2.
\end{align}
One example that satisfies these conditions is $f(u)=-u-u^3$.
Furthermore, we assume that $g(u)$ is continuously differentiable, globally Lipschitz continuous, $|g''(u)|$ is bounded, and satisfies the growth condition, i.e., there exists a constant
$C$ such that
\begin{align}
g&\in C^1,\label{eq20210627_1}\\
|g(a) - g(b)| &\leq C |a-b|,\label{eq20210627_2}\\
|g(a)|^2 &\leq C(1+a^2), \label{gu_grow}\\
|g''(u)|&\le C\label{eq20210627_3}.
\end{align}
Throughout this paper, $C$ denotes a generic positive constant, which
may have different values at different occasions. Under these
assumptions on the drift term and the diffusion term, it is proved in
\cite{chow2002stochastic} that there exists a unique continuous
solution $u(t,\cdot) \in H^1(\mathcal{D})$ and $u_t(t,\cdot)\in
L^2(\mathcal{D})$ in $[0,T]$ such that 
\begin{equation} \label{eq:stability-continuous}
\mathbb{E}\left[\sup_{t \leq T} \left\{ \|u_t(t,x)\|_{L^2}^2 + \|\nabla
u(t,x)\|_{L^2}^2 + \|u(t,x)\|_{L^{q+1}}^{q+1} \right\}\right] < \infty.
\end{equation}

Wave propagation problem arises in many fields of scientific and
engineering applications \cite{Durran1999}, with examples including
the geoscience, petroleum engineering, telecommunication, and the
defense industry etc. The deterministic wave equations have been
extensively investigated in the last few decades. For the second-order
wave equation with nonlinear drift terms, we refer the readers to some
partial differential equation (PDE) papers in \cite{glassey1973blow,reed2006abstract}, which discuss their well-posedness or blow-up
properties based on different polynomial nonlinear drift terms. 
A large variety of numerical methods have been proposed for the
numerical approximations of the second-order wave equation, including
finite difference, finite element, finite volume, spectral methods and
integral equation based methods. We refer to
\cite{adjerid2011discontinuous,baccouch2012local,chou2014optimal,chung2006optimal,falk1999explicit,grote2006discontinuous,monk2005discontinuous,riviere2003discontinuous,safjan1993high,sun2021,xing2013energy,zhong2011numerical}, and the references therein for
various numerical methods based on the Galerkin approach. 

In the wave propagation applications, the stochastic effects are
important in accurately modeling them due to the uncertainty of the
media. The stochastic wave equation is a hyperbolic type stochastic
partial differential equation (SPDE), and its solution behavior is
very different from that of the stochastic heat equation.  
Numerical methods for the wave equation with various forms of
stochasticity have been studied in the literature. For example, the
wave equations with random coefficients or random initial/boundary
conditions were studied in \cite{chou2019energy,gottlieb2008galerkin,motamed2014}, and numerical methods for the stochastic wave equation
with additive noise were presented in \cite{cohen2013trigonometric,cui2019strong,gubinelli2018renormalization,hausenblas2010weak,kovacs2010finite}. For the multiplicative noise, the well-posedness or
the regularity of the solutions of the stochastic wave equation was
considered in \cite{dalang2009stochastic,dalang1998stochastic,millet2000stochastic,millet1999stochastic},
where the nonlinear drift term is Lipschitz continuous, and the noise
is white in time and correlated in space. Later, in
\cite{chow2002stochastic,chow2009nonlinear}, some blow-up solutions
were presented for a class of nonlinear stochastic wave equations with
white in time and correlated in space noise. In
\cite{chow2002stochastic}, the theorems on well-posedness of the local
and global solutions were given for the stochastic wave equation with
nonlinear drift term in the form of \eqref{eq20191102_1}. The
long-time asymptotic bounds of the solutions were proved in
\cite{chow2006asymptotics}. There have also been some studies on
numerical methods for the stochastic wave equation with multiplicative
noise. For example, in \cite{walsh2006numerical}, a difference scheme
was presented for the model when both the nonlinear drift term $f$ and
diffusion term $g$ are Lipschitz continuous, and the optimal rates of
convergence were established for such method. The stochastic wave
equations with Lipschitz continuous nonlinear drift term and diffusion
term were studied using the semi-group approach in
\cite{anton2016full,cohen2018numerical,cohen2015fully,quer2006space}. As a comparison, the multiplicative noise is considered in this paper based on the variational approach (see
\cite{feng2014finite,feng2017finite,feng2020fully,feng2018strong}),
and the nonlinear drift term is not Lipschitz continuous.

In this paper, we present the fully discrete methods for the
stochastic wave equation with multiplicative noise, by utilizing the
finite element approximation in space, and the implicit Euler/modified Crank-Nicolson
approximation in time. There are two main objectives in this paper.
First, we want to establish the second moment and higher order moment
stability results of the discrete solutions in $L^2$ norm and various
energy norms. This is motivated by the conjecture that the numerical
solutions will inherit the stability properties of the strong
solutions. The goal is achieved by designing the proper numerical scheme, choosing the
test function and using an energy argument. Second, we want to provide the convergence rates of the
error estimates in both $L^2$ and
energy norms for the proposed method when the nonlinear drift term is
not Lipschitz continuous and the diffusion term is multiplicative noise, and this would be the first work to achieve such goal, to our best knowledge. To achieve this objective, we rewrite the numerical scheme in the mixed form, utilize the above stability results, establish several H\"{o}lder continuity results, and construct a subset with nearly probability one, to handle
complex interaction between the multiplicative noise and super-linear
nonlinear term, such that the error estimates hold on this subset. This concept is proposed based on the bounds of the solutions and is motivated by the idea in \cite{carelli2012rates}.

The rest of the paper is organized as follows. In Section \ref{sec-2},
we prove the H\"{o}lder continuity results and some technical lemmas which can be used to establish the error estimates. In Section
\ref{sec-3}, we present the fully discrete finite element methods for
the stochastic wave equation, and prove the discrete stability
in $L^2$ norm, discrete stability in various energy norm, as well as
the discrete higher order moment stability in $L^2$ norm and energy
norms. In Section \ref{sec-4}, the error estimates in $L^2$ norm and
energy norms are established on a subset of which the probability approaches $1$ as
spatial mesh size decreases to $0$. Numerical experiments are provided in Section
\ref{sec-5} to validate the theoretical results, including discrete
stability results and convergence rates in different norms.

\section{Preliminaries and properties of SPDE solution}\label{sec-2}
In this section, we first derive the H\"{o}lder continuity in time for
the strong solution $u$, and then present some basic lemmas that will
be used in the stability and error estimate analysis in Sections
\ref{sec-3} and \ref{sec-4}. 

The standard Sobolev notations are adopted in the paper. We use
$\|\cdot\|_{L^p}$ and $\|\cdot\|_{H^k}$ to denote the $L^p$ and $H^k$
norms in the whole domain $\mathcal{D}$.  The notation $(\cdot\, ,
\cdot)$ represents the standard inner product on $\mathcal{D}$. 


\begin{lemma}\label{lem20190928_1}
In the 1D and 2D settings, for each individual term $u^q$ of the
nonlinear function $f(\cdot)$ defined in \eqref{eq20191102_1}, we have
$$ 
\|u^q-v^q\|_{L^2}^2\le
C\left(\|u\|_{H^1}^{2(q-1)}+\|v\|_{H^1}^{2(q-1)}\right)\|u-v\|_{H^1}^2,
$$
where $q>0$ is an integer, and $u$ and $v$ are two real value
functions. The same result holds for $q=1,\,2,\,3$ in the 3D setting.
\end{lemma}
\begin{proof}
In the 1D and 2D settings, we have
\begin{align}\label{eq20190907_5}
\|u^q-v^q\|_{L^2}^2&\le\|(u^{q-1}+u^{q-2}v+\cdots+v^{q-1})(u-v)\|_{L^2}^2\\
&\le C\left(\|u^{q-1}(u-v)\|_{L^2}^2+\|v^{q-1}(u-v)\|_{L^2}^2\right)\notag\\
&\le C\left(\|u\|_{H^1}^{2(q-1)}+\|v\|_{H^1}^{2(q-1)}\right)\|u-v\|_{H^1}^2\notag
\end{align}
by applying Young's inequality and the embedding theorem. In the 3D
setting, the proof is similar to \eqref{eq20190907_5}.
\end{proof}

The following lemma on the discrete summation-by-parts property is
also provided.  The proof is straightforward and skipped here. 
\begin{lemma}\label{lem20180515_1}
Suppose $\{a_n\}_{n=0}^\ell$ and $\{b_n\}_{n=0}^\ell$ are two
sequences of functions. Then
\begin{equation*}
\sum_{n=1}^\ell(a^n-a^{n-1},b^n) =
(a^\ell,b^\ell)-(a^0,b^0)-\sum_{n=1}^\ell(a^{n-1},b^n-b^{n-1}).
\end{equation*}
\end{lemma}

Define a new variable $v$ as $v:=u_t$. Then we prove the H\"older continuity in time for the strong solution $u$ and $v$ with respect to the spatial $L^2$ norm.
\begin{lemma}\label{lem_20201226_1}
Let $u$ be the strong solution to problem \eqref{sac_s}--\eqref{sac_s2}. Then for any $s,t \in [0,T]$ with $s < t$, we have
\begin{align*}
\E \big[ \| u(t)-u(s) \|_{L^2}^2 \big]  \leq C(t-s)^2,
\end{align*}
where
\begin{align*}
C =&\, C\E \big[ \| u_t(0)\|_{L^2}^2\big]+C\E \big[ \int_0^t\| \Delta u(\zeta)\|_{L^2}^2d\zeta\big]+C\E \big[ \int_0^t \|u(\zeta)\|_{L^{2q}}^{2q}d\zeta\big]+C.
\end{align*}
\end{lemma}
\begin{proof}
The SPDE \eqref{sac_s} leads to
\begin{align}\label{eq20201226_1}
u_t(t)-u_t(0)=\int_0^t\Delta ud\zeta+\int_0^tf(u)d\zeta+\int_0^tg(u)dW(\zeta).
\end{align}
Taking the square, the spatial integral, and the expectation on both sides of \eqref{eq20201226_1}, and then using the triangle inequality, the Schwarz inequality, and It\^{o} isometry, we obtain
\begin{align}\label{eq20201226_2}
&\E \big[ \| u_t(t)\|_{L^2}^2\big]\\
&\le C\E \big[ \| u_t(0)\|_{L^2}^2\big]+C\E \big[ \int_{\mathcal{D}}\bigl(\int_0^t \Delta u(\zeta)d\zeta\bigr)^2dx\big]\notag\\
&\qquad+C\E \big[ \int_{\mathcal{D}}\bigl(\int_0^t f(u(\zeta))d\zeta\bigr)^2dx\big]+C\E \big[ \int_{\mathcal{D}}\bigl(\int_0^t g(u(\zeta))dW(\zeta)\bigr)^2dx\big]\notag\\
&\le C\E \big[ \| u_t(0)\|_{L^2}^2\big]+C\E \big[ \int_0^t\| \Delta u(\zeta)\|_{L^2}^2d\zeta\big]\notag\\
&\qquad+C\E \big[ \int_0^t\| f(u(\zeta))\|_{L^2}^2d\zeta\big]+C\E \big[\int_0^t\| g(u(\zeta))\|_{L^2}^2d\zeta\big]\notag\\
&\le C\E \big[ \| u_t(0)\|_{L^2}^2\big]+C\E \big[ \int_0^t\| \Delta u(\zeta)\|_{L^2}^2d\zeta\big]+C\E \big[ \int_0^t \|u(\zeta)\|_{L^{2q}}^{2q}d\zeta\big]+C\notag,
\end{align} 
where \eqref{gu_grow} is used in the derivation of the last inequality. For any $s,t \in [0,T]$ with $s < t$, we have
\begin{align}\label{eq20201226_3}
\E \big[ \| u(t)-u(s) \|_{L^2}^2 \big] =&\, \E \big[ \| u_t(\xi)\|_{L^2}^2 \big] (t-s)^2\\
\le &\, C(t-s)^2\notag,
\end{align}
where $\xi\in(s,t)$ and
\begin{align*}
C =\,& C\E \big[ \| u_t(0)\|_{L^2}^2\big]+C\E \big[ \int_0^t\| \Delta u(\zeta)\|_{L^2}^2d\zeta\big]+C\E \big[ \int_0^t \|u(\zeta)\|_{L^{2q}}^{2q}d\zeta\big]+C.
\end{align*} 
This finishes the proof of the lemma.
\end{proof}

\begin{lemma}\label{lem_20201226_2}
For any $s,t \in [0,T]$ with $s < t$, we have 
\begin{align*}
\E \big[ \| v(t)-v(s) \|_{L^2}^2 \big]  \leq C(t-s),
\end{align*}
where
\begin{align*}
C =C\E \big[ \int_0^t\| \Delta u(\zeta)\|_{L^2}^2d\zeta\big]+C\E \big[ \int_0^t\| u(\zeta)\|_{L^{2q}}^{2q}d\zeta\big]+C.
\end{align*}
\end{lemma}
\begin{proof}
By \eqref{sac_s}, for any $s,t \in [0,T]$ with $s < t$, we have
\begin{align}\label{eq20201226_4}
v(t)-v(s)=\int_s^t\Delta ud\zeta+\int_s^tf(u)d\zeta+\int_s^tg(u)dW(\zeta).
\end{align}
Taking the square, the spatial integral, and the expectation on both sides of \eqref{eq20201226_4}, and then using the triangle inequality, the Schwarz inequality, and It\^{o} isometry, we obtain
\begin{align}\label{eq20201226_5}
&\E \big[ \| v(t)-v(s) \|_{L^2}^2 \big] \\
&\le C\E \big[ \int_s^t\| \Delta u(\zeta)\|_{L^2}^2d\zeta\big](t-s)\notag\\
&\qquad+C\E \big[ \int_s^t\| f(u(\zeta))\|_{L^2}^2d\zeta\big](t-s)+C\E \big[\int_s^t\| g(u(\zeta))\|_{L^2}^2d\zeta\big]\notag,\\
&\le C\E \big[ \int_s^t\| \Delta u(\zeta)\|_{L^2}^2d\zeta\big](t-s)\notag\\
&\qquad+C\E \big[ \int_s^t\| u(\zeta)\|_{L^{2q}}^{2q}d\zeta\big](t-s)+C\E \big[\int_s^t\|u(\zeta)\|_{L^2}^2d\zeta\big]+C(t-s)\notag,\\
&\le C(t-s),\notag
\end{align} 
where
\begin{align*}
C =C\E \big[ \int_s^t\| \Delta u(\zeta)\|_{L^2}^2d\zeta\big]+C\E \big[ \int_s^t\| u(\zeta)\|_{L^{2q}}^{2q}d\zeta\big]+C\E \big[\int_s^t\|u(\zeta)\|_{L^2}^2d\zeta\big]+C,
\end{align*} 
and this finishes the proof of the lemma.
\end{proof}

Next, we prove the H\"older continuity in time for the strong solution $u$ and $v$ with respect 
to the spatial $H^1$-seminorm.

\begin{lemma}\label{lem_20200630_1}
Let $u$ be the strong solution to problem \eqref{sac_s}--\eqref{sac_s2}. Under the assumptions \eqref{eq20210627_1}-\eqref{gu_grow}, for any $s,t \in [0,T]$ with $s < t$, we have
\begin{align*}
\E \big[ \|\nabla (u(t)-u(s)) \|_{L^2}^2 \big]  \leq C(t-s)^2,
\end{align*}
where
\begin{align*}
C =\,& C\E \big[ \|\nabla u_t(0)\|_{L^2}^2\big]+C\E \big[ \int_0^t\|\nabla \Delta u(\zeta)\|_{L^2}^2d\zeta\big] + C\E \big[ \int_0^t \|u(\zeta)\|_{L^{4(q-1)}}^{4(q-1)}d\zeta\big]\\
&+C\E \big[\int_0^t\|\nabla u(\zeta)\|_{L^4}^4d\zeta\big]+C.
\end{align*}
\end{lemma}

\begin{proof}
From the equation \eqref{sac_s}, we get
\begin{align}\label{eq20200629_1}
u_t(t)-u_t(0)=\int_0^t\Delta ud\zeta+\int_0^tf(u)d\zeta+\int_0^tg(u)dW(\zeta).
\end{align}
Taking the gradient, the square, the spatial integral, and the expectation on both sides of \eqref{eq20200629_1}, and then using the triangle inequality, the Schwarz inequality, and It\^{o} isometry, we obtain
\begin{align}\label{eq20200629_2}
&\E \big[ \|\nabla u_t(t)\|_{L^2}^2\big]\\
&\le C\E \big[ \|\nabla u_t(0)\|_{L^2}^2\big]+C\E \big[ \int_{\mathcal{D}}\bigl(\int_0^t\nabla \Delta u(\zeta)d\zeta\bigr)^2dx\big]\notag\\
&\qquad+C\E \big[ \int_{\mathcal{D}}\bigl(\int_0^t\nabla f(u(\zeta))d\zeta\bigr)^2dx\big]+C\E \big[ \int_{\mathcal{D}}\bigl(\int_0^t\nabla g(u(\zeta))dW(\zeta)\bigr)^2dx\big]\notag\\
&\le C\E \big[ \|\nabla u_t(0)\|_{L^2}^2\big]+C\E \big[ \int_0^t\|\nabla \Delta u(\zeta)\|_{L^2}^2d\zeta\big]\notag\\
&\qquad+C\E \big[ \int_0^t\|\nabla f(u(\zeta))\|_{L^2}^2d\zeta\big]+C\E \big[\int_0^t\|\nabla g(u(\zeta))\|_{L^2}^2d\zeta\big]\notag\\
&\le C\E \big[ \|\nabla u_t(0)\|_{L^2}^2\big]+C\E \big[ \int_0^t\|\nabla \Delta u(\zeta)\|_{L^2}^2d\zeta\big]\notag\\
&\qquad+C\E \big[ \int_0^t \|u(\zeta)\|_{L^{4(q-1)}}^{4(q-1)}d\zeta\big]+C\E \big[\int_0^t\|\nabla u(\zeta)\|_{L^4}^4d\zeta\big]+C\notag.
\end{align} 
Therefore, for any $s,t \in [0,T]$ with $s < t$, we have
\begin{align}\label{eq20200629_3}
\E \big[ \|\nabla (u(t)-u(s)) \|_{L^2}^2 \big] =& \E \big[ \|\nabla u_t(\xi)\|_{L^2}^2 \big] (t-s)^2\\
\le & C(t-s)^2\notag,
\end{align}
where $\xi\in(s,t)$ and
\begin{align*}
C =\,& C\E \big[ \|\nabla u_t(0)\|_{L^2}^2\big]+C\E \big[ \int_0^t\|\nabla \Delta u(\zeta)\|_{L^2}^2d\zeta\big] + C\E \big[ \int_0^t \|u(\zeta)\|_{L^{4(q-1)}}^{4(q-1)}d\zeta\big]\\
&+C\E \big[\int_0^t\|\nabla u(\zeta)\|_{L^4}^4d\zeta\big]+C.
\end{align*} 
This finishes the proof of the lemma.
\end{proof}

\begin{lemma}\label{lem_20201226_4}
Let $u$ be the strong solution to problem \eqref{sac_s}--\eqref{sac_s2}. Under the assumptions \eqref{eq20210627_1}-\eqref{gu_grow}, for any $s,t \in [0,T]$ with $s < t$, we have
\begin{align*}
\E \big[ \| \nabla(v(t)-v(s)) \|_{L^2}^2 \big]  \leq C(t-s),
\end{align*}
where
\begin{align*}
C =\,&C\E \big[ \int_s^t\|\nabla\Delta u(\zeta)\|_{L^2}^2d\zeta\big]+C\E \big[ \int_s^t \|u(\zeta)\|_{L^{4(q-1)}}^{4(q-1)}d\zeta\big]\\
&+C\sup_{s \leq \zeta \leq t}\E \big[\|\nabla u(\zeta)\|_{L^4}^4\big]+C.
\end{align*}
\end{lemma}
\begin{proof}
From the SPDE \eqref{sac_s}, for any $s,t \in [0,T]$ with $s < t$, we have
\begin{align}\label{eq20201226_6}
v(t)-v(s)=\int_s^t\Delta ud\zeta+\int_s^tf(u)d\zeta+\int_s^tg(u)dW(\zeta).
\end{align}
Taking the gradient, the square, the spatial integral, and the expectation on both sides of \eqref{eq20200629_1}, and then using the triangle inequality, the Schwarz inequality, and It\^{o} isometry, we obtain
\begin{align}\label{eq20201226_7}
&\E \big[ \| \nabla(v(t)-v(s)) \|_{L^2}^2 \big] \\
&\le C\E \big[ \int_s^t\| \nabla\Delta u(\zeta)\|_{L^2}^2d\zeta\big](t-s)\notag\\
&\qquad+C\E \big[ \int_s^t\| \nabla f(u(\zeta))\|_{L^2}^2d\zeta\big](t-s)+C\E \big[\int_s^t\| \nabla g(u(\zeta))\|_{L^2}^2d\zeta\big]\notag,\\
&\le C\E \big[ \int_s^t\|\nabla \Delta u(\zeta)\|_{L^2}^2d\zeta\big](t-s)\notag\\
&\qquad+C\E \big[ \int_s^t \|u(\zeta)\|_{L^{4(q-1)}}^{4(q-1)}d\zeta\big](t-s)+C\E \big[\int_s^t\|\nabla u(\zeta)\|_{L^4}^4d\zeta\big]+C(t-s)\notag,\\
&\le C(t-s),\notag
\end{align} 
where
\begin{align*}
C =\,&C\E \big[ \int_s^t\|\nabla\Delta u(\zeta)\|_{L^2}^2d\zeta\big]+C\E \big[ \int_s^t \|u(\zeta)\|_{L^{4(q-1)}}^{4(q-1)}d\zeta\big]\\
&+C\sup_{s \leq \zeta \leq t}\E \big[\|\nabla u(\zeta)\|_{L^4}^4\big]+C.
\end{align*} 
This finishes the proof of the lemma.
%
\end{proof}

At the end of this section, we prove the H\"older continuity in time for the strong solution $u$ with respect 
to the spatial $H^2$-seminorm.
\begin{lemma}\label{lem_20200717_1}
Let $u$ be the strong solution to problem \eqref{sac_s}--\eqref{sac_s2}. Under the assumptions \eqref{eq20210627_1}-\eqref{eq20210627_3}, for any $s,t \in [0,T]$ with $s < t$, we have
\begin{align*}
\E \big[ |\nabla^2(u(t)-u(s))|_{L^2}^2 \big]  \leq C(t-s)^2,
\end{align*}
where
\begin{align*}
C =\,& C\E \big[ \|\nabla^2 u_t(0)\|_{L^2}^2\big]+C\E \big[ \int_0^t\|\nabla^2 \Delta u(\zeta)\|_{L^2}^2d\zeta\big]+C\E \big[ \int_0^t \|u(\zeta)\|_{L^{4(q-1)}}^{4(q-1)}d\zeta\big]\\
&+C\E \big[\int_0^t\|\nabla^2 u(\zeta)\|_{L^4}^4d\zeta\big]+C\E \big[\int_0^t\|\nabla u(\zeta)\|_{L^8}^8d\zeta\big]+C.
\end{align*}
\end{lemma}

\begin{proof}
Again, from the equation \eqref{sac_s}, we get
\begin{align}\label{eq20200717_1}
u_t(t)-u_t(0)=\int_0^t\Delta ud\zeta+\int_0^tf(u)d\zeta+\int_0^tg(u)dW(\zeta).
\end{align}
Taking the Hessian, the square, the spatial integral, and the expectation on both sides of \eqref{eq20200629_1}, and then using the triangle inequality, the Schwarz inequality, and It\^{o} isometry, we obtain
\begin{align}\label{eq20200717_2}
&\E \big[ \|\nabla^2 u_t(t)\|_{L^2}^2\big]\\
&\le C\E \big[ \|\nabla^2 u_t(0)\|_{L^2}^2\big]+C\E \big[ \int_{\mathcal{D}}\bigl(\int_0^t\nabla^2 \Delta u(\zeta)d\zeta\bigr)^2dx\big]\notag\\
&\qquad+C\E \big[ \int_{\mathcal{D}}\bigl(\int_0^t\nabla^2 f(u(\zeta))d\zeta\bigr)^2dx\big]+C\E \big[ \int_{\mathcal{D}}\bigl(\int_0^t\nabla^2 g(u(\zeta))dW(\zeta)\bigr)^2dx\big]\notag\\
&\le C\E \big[ \|\nabla^2 u_t(0)\|_{L^2}^2\big]+C\E \big[ \int_0^t\|\nabla^2 \Delta u(\zeta)\|_{L^2}^2d\zeta\big]\notag\\
&\qquad+C\E \big[ \int_0^t\|\nabla^2 f(u(\zeta))\|_{L^2}^2d\zeta\big]+C\E \big[\int_0^t\|\nabla^2 g(u(\zeta))\|_{L^2}^2d\zeta\big]\notag\\
&\le C\E \big[ \|\nabla^2 u_t(0)\|_{L^2}^2\big]+C\E \big[ \int_0^t\|\nabla^2 \Delta u(\zeta)\|_{L^2}^2d\zeta\big]\notag\\
&\qquad+C\E \big[ \int_0^t \|u(\zeta)\|_{L^{4(q-1)}}^{4(q-1)}d\zeta\big]+C\E \big[\int_0^t\|\nabla^2 u(\zeta)\|_{L^4}^4d\zeta\big]\notag\\
&\qquad+C\E \big[\int_0^t\|\nabla u(\zeta)\|_{L^8}^8d\zeta\big]+C\notag.
\end{align} 
Therefore, for any $s,t \in [0,T]$ with $s < t$, we have
\begin{align}\label{eq20200823_1}
\E \big[ \|\nabla^2 (u(t)-u(s)) \|_{L^2}^2 \big] =& \E \big[ \|\nabla^2 u_t(\xi)\|_{L^2}^2 \big] (t-s)^2\\
\le & C(t-s)^2\notag,
\end{align}
where $\xi\in(s,t)$ and
\begin{align*}
C =& C\E \big[ \|\nabla^2 u_t(0)\|_{L^2}^2\big]+C\E \big[ \int_0^t\|\nabla^2 \Delta u(\zeta)\|_{L^2}^2d\zeta\big]+C\E \big[ \int_0^t \|u(\zeta)\|_{L^{4(q-1)}}^{4(q-1)}d\zeta\big]\\
&+C\E \big[\int_0^t\|\nabla^2 u(\zeta)\|_{L^4}^4d\zeta\big]+C\E \big[\int_0^t\|\nabla u(\zeta)\|_{L^8}^8d\zeta\big]+C.
\end{align*} 
This finishes the proof of the lemma.
\end{proof}



\section{Fully discrete finite element methods and stability estimates}\label{sec-3}

In this section, we start by presenting the fully discrete finite
element methods for the stochastic wave equations
\eqref{sac_s}--\eqref{eq20191102_1}, and then establish several
stability estimates of the numerical solutions.  In addition to the second moment stability in $L^2$ norm and energy  norms of the discrete numerical solutions, 
the stability of higher order moments is also provided.

\subsection{Notations and the finite element methods}
Let $\mathcal{T}_h$ be a quasi-uniform triangulation of the domain $\mathcal{D}$.
We consider the $\mathcal{P}_r$-Lagrangian finite element space
\begin{align}\label{eq20180713_1}
V_h = \bigl\{v_h \in C(\bar{\mathcal{D}}): v_h|_{K} \in \mathcal{P}_r(K), \,\, \forall K\in\mathcal{T}_h  	\bigr\},
\end{align}
where $\bar{\mathcal{D}}$ is the closure of the domain  $\mathcal{D}$, $P_r(K)$ denotes the space of all polynomials of degrees up to $r$ on $K$, and 
$r\ge1$ is an integer. Consider a uniform partition of the time domain $[0,\,T]$ with $\tau = T/N$, and denote $t_n=n\tau$ for $n=0,1,\cdots N$.

The fully discretized numerical methods for \eqref{sac_s} is to
seek an $\mathcal{F}_{t_n}$ adapted $V_h$-valued process
$\{u_h^n\}_{n=0}^N$ such that it holds $\mathbb{P}$-almost surely that:
\begin{align}\label{dfem}
&\left(\frac{u^{n+1}_h-2u^{n}_h+u^{n-1}_h}{\tau}, w_h\right) 
+  \tau(\nabla u^{n+1}_h, \nabla w_h) \\ 
&\qquad\qquad\qquad
 = \tau(f^{n+1}_h, w_h)+ (g(u^n_h), w_h) \, \bar{\Delta} W_{n+1} \qquad \forall \, w_h \in V_h \notag,
\end{align}
where the notation $\bar{\Delta}W_{n+1}$ is defined by 
\begin{equation}\label{deltaW}
\bar{\Delta}W_{n+1} :=
W(t_{n+1}) - W(t_n) \sim \mathcal{N}(0,\tau),
\end{equation}
and there are two choices for the discretization of the nonlinear drift term:
\begin{enumerate}
\item Fully implicit discretization: 
\begin{equation} \label{eq:fully-implicit}
f^{n+1}_h:=f(u_h^{n+1}).
\end{equation}
\item Modified Crank-Nicolson discretization: 
\begin{equation}\label{eq:CN}
f_h^{n+1} := \hat{f}(u_h^{n+1},u_h^n) =
\begin{cases}
-\frac{F(u_h^{n+1})-F(u_h^n)}{u_h^{n+1}-u_h^n}\quad &\mathrm{if}\ u_h^{n+1}\neq u_h^n,\\
f(u_h^{n+1})\quad &\mathrm{if}\ u_h^{n+1}=u_h^n,
\end{cases}
\end{equation}
where $F(\cdot)$ is defined in \eqref{eq20191102_2}.
\end{enumerate}

The finite element method \eqref{dfem} involves a two-step implicit
temporal discretization, and would need two initial conditions $u_h^0$
and $u_h^{-1}$ to start. The initial condition $u_h^0=P_h u(x,0)$ is
obtained via a standard $L^2$-projection operator defined as $P_h:
L^2(\D) \longrightarrow V_h$ satisfying
\begin{align*}
\left(P_h u, w_h\right) = (u, w_h) \qquad \forall w_h \in V_h,
\end{align*}
and $u_h^{-1} = u_h^0 - \tau P_h u_t(0,x)$, namely, the backward Euler method is used for the initial step. 
The discrete Laplace operator $\Delta_h: {V}_h\mapsto {V}_h$ is defined
as follows: given $z_h \in {V}_h$, $\Delta_h z_h \in V_h$ is chosen such
that
\begin{equation} \label{eq:discrete-Laplace}
(\Delta_h z_h, w_h)=-(\nabla z_h,\nabla w_h) \qquad \forall\, w_h\in
V_h.
\end{equation}

\subsection{Stability in $L^2$ norm and energy norms}
For the deterministic wave equation (i.e., $g =0$ in \eqref{sac_s}), 
it is well-known that this model preserves the Hamiltonian, defined as  
$$
\mathcal{H}(u) := \frac{1}{2}\|u_t\|_{L^2}^2 + \frac{1}{2}\|\nabla u\|_{L^2}^2 +
(F(u), 1),
$$
where $F(u)$ satisfies the condition \eqref{eq20191102_2}.
The discrete analogue of the Hamiltonian is defined as 
\begin{equation} \label{eq:Hamiltonian-h}
\tilde{\mathcal H}(u_h^n) := \frac{1}{2}\|d_t u_h^n\|_{L^2}^2 + \frac{1}{2}\|\nabla u_h^n\|_{L^2}^2 +(F(u_h^n), 1),
\end{equation}
where $d_t$ denotes the temporal difference operator defined by
\begin{align}\label{eq20180815_3}
d_tu_h^{n}=\frac{u_h^{n}-u_h^{n-1}}{\tau}.
\end{align}

Before stating the stability estimate, we summarize the assumption on the nonlinear drift term below. 
\begin{assumption} \label{assumption}
The nonlinear drift term $f(u)$ given in \eqref{eq20191102_1} satisfies \eqref{eq20191102_2}. 
Furthermore, $F(\cdot)$ is convex if the fully implicit discretization \eqref{eq:fully-implicit} of the nonlinear term is utilized.
\end{assumption} 

\begin{theorem}\label{thm20180815_1}
Let $\{u_h^\ell\}_{\ell=0}^{N}$ denote the numerical solutions of the
finite element methods \eqref{dfem}.  Under the Assumption
\ref{assumption}, the following inequality holds for any integer
$\ell\in[1,N]$,
\begin{align*}
&\frac12\E\left[ \|d_tu_h^{\ell}\|_{L^2}^2\right] +
\frac12\E\left[\|\nabla u_h^{\ell}\|_{L^2}^2\right] +(F(u_h^\ell), 1)\\
&\qquad\qquad\qquad
+\frac14\sum_{n=0}^{\ell-1}\E\left[\|d_tu_h^{n+1}-d_tu_h^n\|_{L^2}^2\right]
+\frac12\sum_{n=0}^{\ell-1}\E\left[\|\nabla(u_h^{n+1}-u_h^n)\|_{L^2}^2\right] \\
&\quad
=\E\left[\tilde{\mathcal H}(u_h^\ell)\right]
+\frac14\sum_{n=0}^{\ell-1}\E\left[\|d_tu_h^{n+1}-d_tu_h^n\|_{L^2}^2\right]
+\frac12\sum_{n=0}^{\ell-1}\E\left[\|\nabla(u_h^{n+1}-u_h^n)\|_{L^2}^2\right] \le C\notag.
\end{align*}
\end{theorem}
\begin{proof}
Taking the test function $w_h=d_tu^{n+1}_h$ in \eqref{dfem}, we have
\begin{align}\label{eq20180815_1}
&\left(\frac{u^{n+1}_h-2u^{n}_h+u^{n-1}_h}{\tau}, d_tu^{n+1}_h\right) 
+ \tau\left(\nabla u^{n+1}_h, \nabla d_tu^{n+1}_h\right)\\
&\hskip2cm
= \tau\left(f^{n+1}_h, d_tu^{n+1}_h\right)+ \left(g(u^n_h), d_tu^{n+1}_h\right) \,
\bar{\Delta} W_{n+1} \notag.
\end{align}
The two terms on the left can be rewritten as:
\begin{align}\label{eq20180816_1}
&\left(\frac{u^{n+1}_h-2u^{n}_h+u^{n-1}_h}{\tau}, d_tu^{n+1}_h\right)
=\left(d_tu_h^{n+1}- d_tu_h^n,d_tu_h^{n+1}\right)\\
&\hskip1.5cm
=\frac12\|d_tu_h^{n+1}\|_{L^2}^2 - \frac12\|d_tu_h^n\|_{L^2}^2
+\frac12 \|d_tu_h^{n+1}-d_tu_h^n\|_{L^2}^2,\notag\\ 
&\tau\left(\nabla u^{n+1}_h, \nabla d_tu^{n+1}_h\right)
=\left(\nabla u^{n+1}_h, \nabla u^{n+1}_h-\nabla u^{n}_h\right) 
\label{eq20180816_2}\\
&\hskip1.5cm
=\frac12\|\nabla u_h^{n+1}\|_{L^2}^2 - \frac12\|\nabla u_h^n\|_{L^2}^2
+\frac12\|\nabla(u_h^{n+1}-u_h^n)\|_{L^2}^2,\notag
\end{align}
Taking the expectation on the last term yields
\begin{align}
\E\left[\left(g(u^n_h), d_tu^{n+1}_h\right) \, \bar{\Delta}W_{n+1}\right]
&=\E\left[\left( g(u_h^n), d_tu_h^{n+1}-d_tu_h^n\right) \,\bar{\Delta} W_{n+1}\right]\label{eq20180816_3}\\
\le C\tau\E&\left[1+\|u_h^n\|_{L^2}^2\right] + \frac14\E\left[\|d_tu_h^{n+1}-d_tu_h^n\|_{L^2}^2\right]\notag,
\end{align}
where the first equality comes from the fact that 
$$ 
  \E\left[\left(g(u^n_h), d_tu^{n}_h\right) \,
  \bar{\Delta}W_{n+1}\right]=\E\left(g(u^n_h),
  d_tu^{n}_h\right)\,\E\left[\bar{\Delta}W_{n+1}\right]=0,
$$
and the second inequality is a result of the Cauchy-Schwarz
inequality, the growth condition of $g(u)$ in \eqref{gu_grow} and the
property of $\bar{\Delta}W_{n+1}$ in \eqref{deltaW}.

The bound of the first term on the right-hand side is discussed case
by case:
\begin{enumerate}
\item For the fully implicit discretization \eqref{eq:fully-implicit},
use Taylor's formula to derive
\begin{align*}
F(u_h^n) = F(u_h^{n+1}) + f(u_h^{n+1})(u_h^{n+1}-u_h^n) +
\frac12F''(\xi)(u_h^{n+1}-u_h^n)^2,
\end{align*}
where $\xi$ locates between $u_h^n$ and $u_h^{n+1}$. Notice that
$\frac12F''(\xi)(u_h^{n+1}-u_h^n)^2\ge0$ under the Assumption \ref{assumption}.

\item For the modified Crank-Nicolson discretization \eqref{eq:CN}, 
\begin{align*}
F(u_h^n) = F(u_h^{n+1}) + \hat{f}(u_h^{n+1}, u_h^n)(u_h^{n+1}-u_h^n).
\end{align*}
\end{enumerate}
Therefore, one can conclude that, for both fully implicit
\eqref{eq:fully-implicit} and modified Crank-Nicolson \eqref{eq:CN}
discretizations, 
\begin{align}\label{eq20181008_1}
\tau(-f^{n+1}_h, d_tu^{n+1}_h ) \ge (F(u_h^{n+1})-F(u_h^n), 1).
\end{align}
Summing the equation \eqref{eq20180815_1} over $n$ from $1$ to
$\ell-1$, taking the expectation on both sides and using the results
\eqref{eq20180816_1}-\eqref{eq20181008_1}, we have
\begin{align}\label{eq20180815_2}
&\E\left[\tilde{\mathcal
  H}(u_h^\ell)\right]+\frac14\sum_{n=1}^{\ell-1}\E\left[\|d_tu_h^{n+1}-d_tu_h^n\|_{L^2}^2\right]+\frac12\sum_{n=1}^{\ell-1}\E\left[\|\nabla(u_h^{n+1}-u_h^n)\|_{L^2}^2\right]\\
&\le \E\left[\tilde{\mathcal H}(u_h^1)\right]+
  C\tau\sum_{n=1}^{\ell-1}\E\left[\|u_h^n\|_{L^2}^2\right]+C.\notag
\end{align}
Applying the Gronwall's inequality yields
\begin{align}\label{eq20180818_1}
\E\left[\tilde{\mathcal H}(u_h^\ell)\right]
+\frac14\sum_{n=0}^{\ell-1}\E\left[\|d_tu_h^{n+1}-d_tu_h^n\|_{L^2}^2\right]
+\frac12\sum_{n=0}^{\ell-1}\E\left[\|\nabla(u_h^{n+1}-u_h^n)\|_{L^2}^2\right] \le C\notag.
\end{align}
This gives the desired stability in $L^2$ norm and energy norms.
\end{proof}


\subsection{Stability of the higher order moments}
The following stability of the higher order moments can be established based on the stability results in Theorem \ref{thm20180815_1}.
\begin{theorem}\label{high_moment1} 
Let $\{u_h^\ell\}_{\ell=0}^{N}$ denote the numerical solutions of the
finite element methods \eqref{dfem}.  Under the Assumption
\ref{assumption}, for any integer $p\ge 2$, it holds for any integer
$\ell\in[1,N]$ that
\begin{equation}\label{high_moment2}
\E\left[\|\nabla u_h^{\ell}\|_{L^2}^p + \|d_t u_h^{\ell}\|_{L^2}^p +
  (F(u_h^\ell),1)^p \right] \le C.
\end{equation}
\end{theorem}
\begin{proof}
To ease the presentation, the proof is divided into three steps.

{\bf Step 1.} Following the results
  \eqref{eq20180815_1}--\eqref{eq20181008_1} in the proof of Theorem
  \ref{thm20180815_1}, we have
\begin{align}\label{eq20181231_1}
&\frac{1}{2}\|d_t u_h^{n+1}\|_{L^2}^2 - \frac{1}{2}\|d_t
u_h^n\|_{L^2}^2 + \frac{1}{2}\|d_tu_h^{n+1} - d_t u_h^n\|_{L^2}^2 \\
& \hskip1cm
+\frac{1}{2}\|\nabla u_h^{n+1}\|_{L^2}^2 - \frac{1}{2}\|\nabla
u_h^n\|_{L^2}^2 + \frac{1}{2}\|\nabla u_h^{n+1} - \nabla
u_h^n\|_{L^2}^2 \notag \\
& \hskip1cm
+ (F(u_h^{n+1}) - F(u_h^n), 1) \leq (g(u_h^n), d_t
u_h^{n+1})\bar{\Delta}W_{n+1}, \notag
\end{align}
which can be recast as follows thanks to the definition
\eqref{eq:Hamiltonian-h},
\begin{align} \label{eq20190909_1}
&\tilde{\mathcal H}(u_h^{n+1}) - \tilde{\mathcal H}(u_h^n) + \frac{1}{2}\|d_tu_h^{n+1} - d_t
u_h^n\|_{L^2}^2  + \frac{1}{2}\|\nabla u_h^{n+1} - \nabla 
u_h^n\|_{L^2}^2 \\ 
& \hskip3cm
\leq \left(g(u_h^n), d_t u_h^{n+1}\right)\bar{\Delta}W_{n+1}. \notag
\end{align}
Utilizing the following identity 
$$ 
\tilde{\mathcal H}(u_h^{n+1}) + \frac{1}{2} \tilde{\mathcal H}(u_h^{n})= 
\frac{3}{4}\left(\tilde{\mathcal H}(u_h^{n+1}) + \tilde{\mathcal H}(u_h^{n})\right)  + 
\frac{1}{4}\left(\tilde{\mathcal H}(u_h^{n+1}) - \tilde{\mathcal H}(u_h^{n})\right),
$$ 
and multiplying \eqref{eq20190909_1} by the term $\tilde{\mathcal H}(u_h^{n+1}) +
\frac{1}{2} \tilde{\mathcal H}(u_h^{n})$, we obtain
\begin{align} \label{eq20190909_2}
& \frac34\left( \tilde{\mathcal H}(u_h^{n+1})^2 - \tilde{\mathcal  H}(u_h^{n})^2 \right) 
+\frac14\left( \tilde{\mathcal H}(u_h^{n+1}) - \tilde{\mathcal H}(u_h^{n}) \right)^2\\
& \hskip0.5cm
+\frac{1}{2}\left(\|d_tu_h^{n+1} - d_t u_h^n\|_{L^2}^2  + \|\nabla u_h^{n+1} - \nabla u_h^n\|_{L^2}^2\right)
\left( \tilde{\mathcal H}(u_h^{n+1}) + \frac{1}{2} \tilde{\mathcal H}(u_h^{n})\right) \notag\\
& \hskip0.5cm
 \leq \left(g(u_h^n), d_t u_h^{n+1}\right)\bar{\Delta}W_{n+1}
\left( \tilde{\mathcal H}(u_h^{n+1}) + \frac{1}{2} \tilde{\mathcal H}(u_h^{n})\right).	\notag
\end{align}
The right-hand side of \eqref{eq20190909_2} can be rewritten as 
\begin{align} \label{eq20190909_3}
&
 \left(g(u_h^n), d_t u_h^{n+1}\right)\bar{\Delta}W_{n+1} 
 \left(\tilde{\mathcal H}(u_h^{n+1}) + \frac{1}{2} \tilde{\mathcal H}(u_h^{n})\right) \\
& \hskip0.5cm
 =  \left(g(u_h^n), d_t u_h^{n+1} - d_t u_h^n\right)\bar{\Delta}W_{n+1}  
\left(\tilde{\mathcal H}(u_h^{n+1}) + \frac{1}{2} \tilde{\mathcal H}(u_h^{n})\right)
\notag\\
& \hskip0.5cm \qquad 
+ \left(g(u_h^n), d_t u_h^n\right)\bar{\Delta}W_{n+1} 
\left(\tilde{\mathcal H}(u_h^{n+1}) + \frac{1}{2} \tilde{\mathcal H}(u_h^{n})\right) \notag \\
& \hskip0.5cm 
\leq \bigg( \frac{1}{4}\|d_t u_h^{n+1} - d_t u_h^n\|_{L^2}^2 +
C\left(\|u_h^n\|_{L^2}^2+1\right)(\bar{\Delta}W_{n+1})^2 \notag \\
& \hskip0.5cm \qquad 
+  \left(g(u_h^n), d_t u_h^n\right)\bar{\Delta}W_{n+1}\bigg)
\left( \tilde{\mathcal H}(u_h^{n+1}) + \frac{1}{2} \tilde{\mathcal H}(u_h^{n})\right),  \notag
\end{align}
by applying the Cauchy-Schwarz inequality and the growth condition of $g(u)$ in \eqref{gu_grow}.
The last two terms can be bounded as
\begin{align} \label{eq20190909_4}
&
\bigg(C(\|u_h^n\|_{L^2}^2+1)(\bar{\Delta} W_{n+1})^2+  \left(g(u_h^n), d_t u_h^n\right)\bar{\Delta}W_{n+1}\bigg)\\
&\times\left( \tilde{\mathcal H}(u_h^{n+1}) + \frac{1}{2} \tilde{\mathcal H}(u_h^{n})\right) \notag\\
& 
\leq \frac18 \left(\tilde{\mathcal H}(u_h^{n+1}) - \tilde{\mathcal H}(u_h^n)\right)^2 
+C(\|u_h^n\|_{L^2}^4+1)(\bar{\Delta}W_{n+1})^4\notag\\
&+C\|d_tu_h^n\|_{L^2}^2 (\|u_h^n\|_{L^2}^2+1) (\bar{\Delta}W_{n+1})^2+C\tilde{\mathcal H}(u_h^n)(\|u_h^n\|_{L^2}^2+1) (\bar{\Delta}W_{n+1})^2	\notag \\
& \hskip5mm
+ \frac{3}{2}\tilde{\mathcal H}(u_h^n)(g(u_h^n), d_t u_h^n)\bar{\Delta}W_{n+1}.		\notag
\end{align}
Combining the equation \eqref{eq20190909_2} with the results in \eqref{eq20190909_3}-\eqref{eq20190909_4} yields
\begin{align}\label{eq20200306_1}
& \frac34\left( \tilde{\mathcal H}(u_h^{n+1})^2 - \tilde{\mathcal H}(u_h^{n})^2 \right)
+\frac18\left( \tilde{\mathcal H}(u_h^{n+1}) - \tilde{\mathcal H}(u_h^{n}) \right)^2\\
& \hskip0.1cm \quad
+\left(\frac{1}{4}\|d_tu_h^{n+1} - d_t u_h^n\|_{L^2}^2  +
\frac{1}{2}\|\nabla u_h^{n+1} - \nabla u_h^n\|_{L^2}^2\right)
\left( \tilde{\mathcal H}(u_h^{n+1}) + \frac{1}{2} \tilde{\mathcal H}(u_h^{n})\right)
\notag \\
& \hskip0.1cm  
\leq C(\|u_h^n\|_{L^2}^4+1)(\bar{\Delta}W_{n+1})^4 
+C\|d_tu_h^n\|_{L^2}^2 (\|u_h^n\|_{L^2}^2+1) (\bar{\Delta}W_{n+1})^2 	\notag \\
& \hskip0.1cm \quad 
+C\tilde{\mathcal H}(u_h^n)(\|u_h^n\|_{L^2}^2+1) (\bar{\Delta}W_{n+1})^2
+ \frac{3}{2}\tilde{\mathcal H}(u_h^n)(g(u_h^n), d_t u_h^n)\bar{\Delta}W_{n+1}.	 \notag
\end{align}
Summing the equation \eqref{eq20200306_1} over $n$ from $1$ to $\ell-1$ and taking expectation on both sides, we obtain 
\begin{align} \label{eq20190909_6}
& \frac{3}{4}\E\left[\tilde{\mathcal H}(u_h^\ell)^2\right] +
\frac{1}{8}\sum_{n=1}^{\ell-1} \E \left[ 
[\tilde{\mathcal H}(u_h^{n+1}) - \tilde{\mathcal H}(u_h^n)]^2 \right] + \sum_{n=1}^{\ell-1}
\E \bigg[\bigg(\frac{1}{4}\|d_tu_h^{n+1} \\ 
&  \quad - d_t u_h^n\|_{L^2}^2  +
\frac{1}{2}\|\nabla u_h^{n+1} - \nabla u_h^n\|_{L^2}^2\bigg)	\left( \tilde{\mathcal H}(u_h^{n+1}) + \frac{1}{2} \tilde{\mathcal H}(u_h^{n})\right)\bigg]		\notag\\
& \leq \frac{3}{4}\E\left[\tilde{\mathcal H}(u_h^0)^2\right] 
+ C\tau^2 \sum_{n=1}^{\ell-1}\E\left[ (\|u_h^n\|_{L^2}^4+1) \right] 	\notag  \\
& \quad 
+ C\tau \sum_{n=1}^{\ell-1}\E\left[ \|d_tu_h^n\|_{L^2}^2 (\|u_h^n\|_{L^2}^2+1) \right]+ C\tau \sum_{n=1}^{\ell-1}\E\left[\tilde{\mathcal H}(u_h^n)(\|u_h^n\|_{L^2}^2+1)\right].\notag 
\end{align}
Following the definition of $\tilde{\mathcal H}(u_h^n)$ in \eqref{eq:Hamiltonian-h}, one has 
 $\tilde{\mathcal H}(u_h^n) \geq \frac{1}{2} \max\left(\|u_h^n\|_{L^2}^2, \|d_tu_h^n\|_{L^2}^2\right)$, which implies that 
\begin{align*}
& \|u_h^n\|_{L^2}^4 + 1 \leq C \tilde{\mathcal H}(u_h^n)^2 + 1, \quad  \\
&\|d_tu_h^n\|_{L^2}^2 (\|u_h^n\|_{L^2}^2 +1) \leq C \|d_tu_h^n\|_{L^2}^4 + \left(\|u_h^n\|_{L^2}^4 +1\right) \leq C \tilde{\mathcal H}(u_h^n)^2 + 1, \quad  \\
& \tilde{\mathcal H}(u_h^n) (\|u_h^n\|_{L^2}^2+1) \leq C \tilde{\mathcal H}(u_h^n)^2 + \left(\|u_h^n\|_{L^2}^4 +1\right) \leq C \tilde{\mathcal H}(u_h^n)^2 + 1.
\end{align*}
Therefore, the following result can be obtained by applying Gronwall's inequality: 
\begin{align} \label{eq20190909_7}
& \E\left[\tilde{\mathcal H}(u_h^\ell)^2\right] 
+\sum_{n=1}^{\ell-1} \E \left[ 
\left(\tilde{\mathcal H}(u_h^{n+1}) - \tilde{\mathcal H}(u_h^n)\right)^2 \right] \\ 
& + \sum_{n=0}^{\ell-1}
\E \left[  \left(\|d_tu_h^{n+1} - d_t u_h^n\|_{L^2}^2  + \|\nabla u_h^{n+1} - \nabla u_h^n\|_{L^2}^2 \right)
\left( \tilde{\mathcal H}(u_h^{n+1}) + \tilde{\mathcal H}(u_h^{n}) \right)
\right]  \notag\\
& \leq C, 	\notag 
\end{align}
which gives us the second order moment stability (i.e., \eqref{high_moment2} when $p=2$).

\smallskip
{\bf Step 2.} Next, the higher order moment stability \eqref{high_moment2} can be established for $p=2^m$ with any positive integer $m$. 
Similar to the derivation in Step 1, we start from the equation \eqref{eq20200306_1} and multiply it by
$\tilde{\mathcal H}(u_h^{n+1})^2+\frac12\tilde{\mathcal H}(u_h^n)^2$. After some simple algebra, 
we can obtain the fourth moment stability of the numerical
solution $u_h^\ell$ below
\begin{align} \label{eq20191025_1}
& \E\left[\tilde{\mathcal H}(u_h^\ell)^4\right] 
+\sum_{n=1}^{\ell-1} \E \left[ 
\left(\bigl(\tilde{\mathcal H}(u_h^{n+1})\bigr)^2 - \bigl(\tilde{\mathcal H}(u_h^n)\bigr)^2\right)^2 \right] 
 \leq C. 
\end{align}

Applying this process repeatedly, the $2^m$-th moment stability of the numerical solution $u_h^\ell$ can be obtained for any positive integer $m$.

\smallskip
{\bf Step 3.} For arbitrary positive integer $p$, suppose $2^{m-1}\le
p\le 2^m$ for some $m$, and one can apply the Young's inequality to
obtain
\begin{align}\label{eq20181231_12}
\E\left[\tilde{\mathcal H}(u_h^\ell)^p\right]&\le \E\left[ 
\tilde{\mathcal H}(u_h^\ell)^{2^m}\right]+C \leq C,
\end{align}
where the second inequality follows from the results in {\bf Step 2}.
The proof is done.
\end{proof}

\section{Error estimates}\label{sec-4}
In this section, we present the error estimates of the proposed finite
element methods. The stability estimates studied in the previous
section are crucial in the analysis. 
\subsection{Error equations in mixed form}
Denoting $v = u_t$, we can rewrite the SPDE \eqref{sac_s} as 
\begin{equation} \label{sac_s_mixed}
  \left\{
  \begin{aligned}
    du &= v dt, \\
    dv &= \Delta u dt + f(u)dt + g(u)dW(t).
  \end{aligned}
  \right.
\end{equation}
Define the numerical errors by $e_u^n := u(t_n)-u_h^n :=
\eta_u^n + \xi_u^n$, $e_v^n := v(t_n) - v_h^n :=
\eta_v^n + \xi_v^n$, where
\begin{align*}
\eta_u^n := u(t_n) - P_h u(t_n) \quad \text{and} \quad \xi_u^n := P_h
  u(t_n) - u_h^n,  \quad n = 0,1,2,...,N, \\
\eta_v^n := v(t_n) - P_h v(t_n) \quad \text{and} \quad \xi_v^n := P_h
  v(t_n) - v_h^n,  \quad n = 0,1,2,...,N, 
\end{align*}
represent the errors of the $L^2$-projection and the errors between the
numerical solution and projected strong solution, respectively.

It follows from \eqref{sac_s_mixed} that for all $t_n$, there holds
$\mathbb{P}$-almost surely 
\begin{align}
  & (u(t_{n+1}) - u(t_n), w_h) = \int_{t_n}^{t_{n+1}} (v(s), w_h) ds
  \qquad\qquad\qquad\quad \forall w_h \in V_h,  \label{eq20210627_11} \\
  & (v(t_{n+1}) - v(t_n), z_h) + \int_{t_n}^{t_{n+1}} (\nabla u(s),
  \nabla z_h) ds   \label{eq20210627_12}\\ 
  & \qquad = \int_{t_n}^{t_{n+1}}(f(u(s)), z_h) ds +
  \int_{t_n}^{t_{n+1}} (g(u(s)), z_h) dW(s) \quad \forall z_h\in V_h. \notag
\end{align}
Also, the numerical scheme \eqref{dfem} can be rewritten in the equivalent mixed form as
\begin{align}
&(u^{n+1}_h-u^n_h,w_h) = \tau (v^{n+1}_h,w_h) \qquad \forall w_h \in V_h,\label{eq20210623_1}\\
&\left(v^{n+1}_h - v^n_h, z_h\right) 
+  \tau(\nabla u^{n+1}_h, \nabla z_h) \label{eq20210623_2}\\ 
&\qquad\qquad\qquad
 = \tau(f^{n+1}_h, z_h)+ (g(u^n_h), z_h) \, \bar{\Delta} W_{n+1} \qquad \forall \, z_h \in V_h.\notag
\end{align}
From the properties of the $L^2$-projection, we have $(\eta_u^n, w_h)
= 0$ and $(\eta_v^n, z_h) = 0$. Combining with the equations \eqref{eq20210627_11}-\eqref{eq20210623_2} leads to 
\begin{align} 
  & (\xi_u^{n+1} - \xi_u^n, w_h) = \int_{t_n}^{t_{n+1}} (v(s) -
  v_h^{n+1}, w_h) ds \qquad \forall w_h \in V_h,
  \label{eq:err_u} \\
  &  (\xi_v^{n+1} - \xi_v^{n}, z_h) + \int_{t_n}^{t_{n+1}}(\nabla u(s)
  - \nabla u_h^{n+1}, \nabla z_h) ds \label{eq:err_w}\\
 & \qquad = \int_{t_n}^{t_{n+1}}  ( f(u(s)) - f_h^{n+1}, z_h) ds
  \notag \\ 
  & \qquad\quad + \int_{t_n}^{t_{n+1}} (g(u(s)) - g(u_h^n), z_h) dW(s)
  \qquad\quad \forall
  z_h \in V_h.\notag
\end{align}

\subsection{Error estimates}

To handle the nonlinearity, we define a sequence of sets as
\begin{align} \label{eq:subset-omega}
\tilde{\Omega}_{\kappa,m} = \Bigl\{\omega\in\Omega:
  \max\limits_{1\leq n \leq m}\|u_h^{n}\|_{H^1}^2+\max\limits_{s\leq t_m}\|u(s)\|_{H^1}^2\le \kappa\Bigr\}.
\end{align}
Here $u(t_n)$ is the strong solution of \eqref{sac_s}-\eqref{sac_s2}, $u_h^{n}$ is the numerical solution of \eqref{dfem}, and $\kappa > 0$ will be specified. Clearly, it holds that $\tilde{\Omega}_{\kappa, 0} \supset \tilde{\Omega}_{\kappa,1} \supset \cdots \supset \tilde{\Omega}_{\kappa,\ell}$. 
The following lemma about the nonlinear term is needed in the proof of the error estimates.
\begin{lemma}\label{lem_20201226_5}
Let $u$ be the strong solution to problem \eqref{sac_s}--\eqref{sac_s2}. Under the assumptions \eqref{eq20210627_1}-\eqref{gu_grow}, for any $s,t \in [0,T]$ with $s < t\le t_m$, we have
\begin{align*}
\E \big[\mathds{1}_{\tilde{\Omega}_{\kappa, m}} \| f(u(t))-f(u(s)) \|_{L^2}^2 \big]\le C\kappa^{q-1}(t-s)^2.
\end{align*}
where
\begin{align*}
C =&\, C\E \big[ \| u_t(0)\|_{H^1}^2\big]+C\E \big[ \int_0^t\| \Delta u(\zeta)\|_{H^1}^2d\zeta\big]+C\E \big[ \int_0^t \|u(\zeta)\|_{L^{2q}}^{2q}d\zeta\big]\\
& + C\E \big[ \int_0^t \|u(\zeta)\|_{L^{4(q-1)}}^{4(q-1)}d\zeta\big]+C\E \big[\int_0^t\|\nabla u(\zeta)\|_{L^4}^4d\zeta\big]+C.
\end{align*}
\end{lemma}
\begin{proof}
By Lemma \ref{lem20190928_1}, we obtain
\begin{align}\label{lem_20201227_1}
&\E \big[\mathds{1}_{\tilde{\Omega}_{\kappa, m}}\| f(u(t))-f(u(s)) \|_{L^2}^2 \big]\\
&\le C\E\Bigl[\mathds{1}_{\tilde{\Omega}_{\kappa, m}}\sum_{j=1}^q\left(\|u(t)\|_{H^1}^{2(j-1)}+\|u(s)\|_{H^1}^{2(j-1)}+1\right)\|u(t)-u(s)\|_{H^1}^2\Bigr] \notag\\
&\le C \kappa^{q-1}\E\Bigl[\mathds{1}_{\tilde{\Omega}_{\kappa, m}}\|u(t)-u(s)\|_{H^1}^2\Bigr]\notag.
\end{align}
Utilizing the conclusions in Lemmas \ref{lem_20201226_1} and \ref{lem_20200630_1} yields
\begin{align}\label{lem_20201227_2}
\E \big[ \mathds{1}_{\tilde{\Omega}_{\kappa, m}}\| f(u(t))-f(u(s)) \|_{L^2}^2 \big]\le C\kappa^{q-1}(t-s)^2,
\end{align}
where 
\begin{align*}
C =&\,C\E \big[ \| u_t(0)\|_{H^1}^2\big]+C\E \big[ \int_0^t\| \Delta u(\zeta)\|_{H^1}^2d\zeta\big]+C\E \big[ \int_0^t \|u(\zeta)\|_{L^{2q}}^{2q}d\zeta\big]\\
& + C\E \big[ \int_0^t \|u(\zeta)\|_{L^{4(q-1)}}^{4(q-1)}d\zeta\big]+C\E \big[\int_0^t\|\nabla u(\zeta)\|_{L^4}^4d\zeta\big]+C,
\end{align*}
which finishes the proof of this lemma.
\end{proof}
Next, we state the following main theorem on the error estimate.
\begin{theorem}\label{thm20180815_2}
Let $\{u_h^\ell\}_{\ell=0}^{N}$ denote the numerical solutions of the
finite element methods \eqref{dfem}.  Under the Assumption
\ref{assumption}, the following error estimate holds for any integer
$\ell\in[1,N]$:
\begin{align*}
 &\mathbb{E}\left[\mathds{1}_{\tilde{\Omega}_{\kappa, \ell}}\|e_u^{\ell}\|_{L^2}^2\right]+\mathbb{E}\left[\mathds{1}_{\tilde{\Omega}_{\kappa, \ell}}\|\nabla e_u^{\ell}\|_{L^2}^2 ds\right]+\mathbb{E}\left[\mathds{1}_{\tilde{\Omega}_{\kappa, \ell}}\|e_v^{\ell}\|_{L^2}^2\right]\\ 
 &\le(Ch^2|\ln h| + C\tau+C\tau^2|\ln h|)h^{-\beta},\notag
\end{align*}
where $\beta>0$ can be chosen to be small enough, $\kappa$ satisfies $\kappa^{q-1}=C\ln(h^{-\beta})$, and $\P\left[\tilde{\Omega}_{\kappa,\ell}\right]\rightarrow1\ \text{as}\ h\rightarrow0$.
\end{theorem}

\begin{proof} 
By taking $w_h = \xi_u^{n+1}$ in \eqref{eq:err_u} and $z_h = \xi_v^{n+1}$ in \eqref{eq:err_w}, multiplying \eqref{eq:err_u} and \eqref{eq:err_w} by $\mathds{1}_{\tilde{\Omega}_{\kappa, n+1}}$, and then taking the expectation, we obtain 
\begin{align} 
  & \mathbb{E}\left[\mathds{1}_{\tilde{\Omega}_{\kappa, n+1}}(\xi_u^{n+1} - \xi_u^n, \xi_u^{n+1})\right] = \mathbb{E}\left[\mathds{1}_{\tilde{\Omega}_{\kappa, n+1}}\int_{t_n}^{t_{n+1}} (v(s) -
  v_h^{n+1}, \xi_u^{n+1}) ds\right], \label{eq20210605_1}\\
  &  \mathbb{E}\left[\mathds{1}_{\tilde{\Omega}_{\kappa, n+1}}(\xi_v^{n+1} - \xi_v^{n}, \xi_v^{n+1})\right]\label{eq20210605_2}\\
   & \quad+ \mathbb{E}\left[\mathds{1}_{\tilde{\Omega}_{\kappa, n+1}}\int_{t_n}^{t_{n+1}}(\nabla u(s)
  - \nabla u_h^{n+1}, \nabla \xi_v^{n+1}) ds\right]\notag\\
 &  = \mathbb{E}\left[\mathds{1}_{\tilde{\Omega}_{\kappa, n+1}}\int_{t_n}^{t_{n+1}}  ( f(u(s)) - f_h^{n+1}, \xi_v^{n+1}) ds\right]
  \notag \\ 
  & \quad + \mathbb{E}\left[\mathds{1}_{\tilde{\Omega}_{\kappa, n+1}}\int_{t_n}^{t_{n+1}} (g(u(s)) - g(u_h^n), \xi_v^{n+1}) dW(s)\right]. \notag
\end{align}

The left-hand side of \eqref{eq20210605_1} can be bounded by
\begin{align} \label{eq20210605_3}
& \mathbb{E}\left[\mathds{1}_{\tilde{\Omega}_{\kappa, n+1}}(\xi_u^{n+1} - \xi_u^n, \xi_u^{n+1})\right] \\
= &\, \frac12\mathbb{E}\left[\mathds{1}_{\tilde{\Omega}_{\kappa, n+1}}\|\xi_u^{n+1}\|_{L^2}^2\right] - \frac12\mathbb{E}\left[\mathds{1}_{\tilde{\Omega}_{\kappa, n+1}}\|\xi_u^n\|_{L^2}^2\right] + \frac12\mathbb{E}\left[\mathds{1}_{\tilde{\Omega}_{\kappa, n+1}}\|\xi_u^{n+1}-\xi_u^n\|_{L^2}^2\right]\notag\\
= &\,\frac12\mathbb{E}\left[\mathds{1}_{\tilde{\Omega}_{\kappa, n+1}}\|\xi_u^{n+1}\|_{L^2}^2\right] - \frac12\mathbb{E}\left[\mathds{1}_{\tilde{\Omega}_{\kappa, n}}\|\xi_u^n\|_{L^2}^2\right] + \frac12\mathbb{E}\left[\mathds{1}_{\tilde{\Omega}_{\kappa, n+1}}\|\xi_u^{n+1}-\xi_u^n\|_{L^2}^2\right]\notag\\
&\quad+ \frac12\mathbb{E}\left[(\mathds{1}_{\tilde{\Omega}_{\kappa, n}}-\mathds{1}_{\tilde{\Omega}_{\kappa, n+1}})\|\xi_u^n\|_{L^2}^2\right]\notag\\
\ge&\,\frac12\mathbb{E}\left[\mathds{1}_{\tilde{\Omega}_{\kappa, n+1}}\|\xi_u^{n+1}\|_{L^2}^2\right] - \frac12\mathbb{E}\left[\mathds{1}_{\tilde{\Omega}_{\kappa, n}}\|\xi_u^n\|_{L^2}^2\right] + \frac12\mathbb{E}\left[\mathds{1}_{\tilde{\Omega}_{\kappa, n+1}}\|\xi_u^{n+1}-\xi_u^n\|_{L^2}^2\right]\notag.
\end{align}
By Lemma \ref{lem_20201226_2}, the right-hand side of \eqref{eq20210605_1} is bounded as
\begin{align}
&\mathbb{E}\left[\mathds{1}_{\tilde{\Omega}_{\kappa, n+1}}\int_{t_n}^{t_{n+1}} (v(s) -
  v_h^{n+1}, \xi_u^{n+1}) ds\right]\label{eq20210605_4}\\
=\,&\mathbb{E}\left[\mathds{1}_{\tilde{\Omega}_{\kappa, n+1}}\int_{t_n}^{t_{n+1}} (v(s) -
  v(t^{n+1}), \xi_u^{n+1}) ds\right]\notag\\
  &\quad+\mathbb{E}\left[\mathds{1}_{\tilde{\Omega}_{\kappa, n+1}}\int_{t_n}^{t_{n+1}} (\eta_v^{n+1} + \xi_v^{n+1}), \xi_u^{n+1}) ds\right]\notag\\
 \le\, &\tau\mathbb{E}\left[\mathds{1}_{\tilde{\Omega}_{\kappa, n+1}} \|\xi_u^{n+1}\|_{L^2}^2 \right]+\tau\mathbb{E}\left[\mathds{1}_{\tilde{\Omega}_{\kappa, n+1}} \|\xi_v^{n+1}\|_{L^2}^2 \right] + C\tau^2\notag.
\end{align}

Following the derivation of the inequality \eqref{eq20210605_3}, the first term on the left-hand side of \eqref{eq20210605_2} can be bounded by
\begin{align} \label{eq20210605_5}
& \mathbb{E}\left[\mathds{1}_{\tilde{\Omega}_{\kappa, n+1}}(\xi_v^{n+1} - \xi_v^n, \xi_v^{n+1})\right] \\
\ge&\,\frac12\mathbb{E}\left[\mathds{1}_{\tilde{\Omega}_{\kappa, n+1}}\|\xi_v^{n+1}\|_{L^2}^2\right] - \frac12\mathbb{E}\left[\mathds{1}_{\tilde{\Omega}_{\kappa, n}}\|\xi_v^n\|_{L^2}^2\right] + \frac12\mathbb{E}\left[\mathds{1}_{\tilde{\Omega}_{\kappa, n+1}}\|\xi_v^{n+1}-\xi_v^n\|_{L^2}^2\right]\notag.
\end{align}
The second term on the left-hand side of \eqref{eq20210605_2} is
\begin{align} \label{eq20210605_7}
&\mathbb{E}\left[\mathds{1}_{\tilde{\Omega}_{\kappa, n+1}}\int_{t_n}^{t_{n+1}}(\nabla u(s)
  - \nabla u_h^{n+1}, \nabla \xi_v^{n+1}) ds\right]\\
=&\,\mathbb{E}\left[\mathds{1}_{\tilde{\Omega}_{\kappa, n+1}}\int_{t_n}^{t_{n+1}}(\nabla u(s)
  - \nabla u(t_{n+1}), \nabla \xi_v^{n+1}) ds\right] \notag\\
 &+ \mathbb{E}\left[\mathds{1}_{\tilde{\Omega}_{\kappa, n+1}}\int_{t_n}^{t_{n+1}}(\nabla e_u^{n+1}, \nabla \xi_v^{n+1}) ds\right].\notag
\end{align}
Notice that
\begin{align} \label{eq20210605_6}
\nabla \xi_v^{n+1} &=  \nabla P_hv(t_{n+1}) - \nabla v_h^{n+1}\\
&=(\nabla(P_hu_t(t_{n+1})) -\nabla(d_tu(t_{n+1})))+( \nabla(d_tu(t_{n+1}))- \nabla(d_tu_h^{n+1}))\notag.
\end{align}
For the first term on the right-hand side of \eqref{eq20210605_7}, we can move it to the right-hand side of \eqref{eq20210605_2} by adding a negative sign, and then bound it by 
\begin{align} \label{eq20210605_8}
&-\mathbb{E}\left[\mathds{1}_{\tilde{\Omega}_{\kappa, n+1}}\int_{t_n}^{t_{n+1}}(\nabla u(s)
  - \nabla u(t_{n+1}), \nabla \xi_v^{n+1}) ds\right]\\
\le&\,C\mathbb{E}\left[\mathds{1}_{\tilde{\Omega}_{\kappa, n+1}}\|\nabla u(s) - \nabla u(t_{n+1})\|_{L^2}^2\right] \notag\\
&+ \tau^2\mathbb{E}\left[\mathds{1}_{\tilde{\Omega}_{\kappa, n+1}}\|\nabla(P_hu_t(t_{n+1})) - \nabla(d_tu(t_{n+1}))\|_{L^2}^2\right]\notag\\
&+\frac14\tau^2\mathbb{E}\left[\mathds{1}_{\tilde{\Omega}_{\kappa, n+1}}\| \nabla(d_tu(t_{n+1}))- \nabla(d_tu_h^{n+1})\|_{L^2}^2\right]\notag\\
\le& C\tau^2 + \frac14 \mathbb{E}\left[\mathds{1}_{\tilde{\Omega}_{\kappa, n+1}}\|\nabla e_u^{n+1}-\nabla e_u^n\|_{L^2}^2\right]\notag,
\end{align}
where the triangle inequality, the $H^1$ stability of the $L^2$-projection, Lemma \ref{lem_20200630_1}, and the proof of Lemma \ref{lem_20200630_1} are used in the derivation of the last inequality.
For the second term on the right-hand side of \eqref{eq20210605_7}, we again move the first term in \eqref{eq20210605_6} to the right-hand side of \eqref{eq20210605_2} by adding a negative sign, and obtain
\begin{align} \label{eq20210606_1}
&-\mathbb{E}\left[\mathds{1}_{\tilde{\Omega}_{\kappa, n+1}}\int_{t_n}^{t_{n+1}}(\nabla e_u^{n+1}, \nabla(P_hu_t(t_{n+1})) - \nabla(d_tu(t_{n+1})) ds\right]\\
\le&\, C\tau\mathbb{E}\left[\mathds{1}_{\tilde{\Omega}_{\kappa, n+1}}\|\nabla \xi_u^{n+1}\|_{L^2}^2\right] + C\tau h^2 \mathbb{E}\left[\mathds{1}_{\tilde{\Omega}_{\kappa, n+1}}\|u(t_{n+1})\|_{H^2}^2\right] \notag\\
&+ \tau\mathbb{E}\left[\mathds{1}_{\tilde{\Omega}_{\kappa, n+1}}\|\nabla(P_hu_t(t_{n+1})) - \nabla(P_hd_tu(t_{n+1}))\|_{L^2}^2\right]\notag\\
&+\tau\mathbb{E}\left[\mathds{1}_{\tilde{\Omega}_{\kappa, n+1}}\|\nabla(P_hd_tu(t_{n+1})) -\nabla(d_tu(t_{n+1})\|_{L^2}^2\right]\notag\\
\le&\, C\tau\mathbb{E}\left[\mathds{1}_{\tilde{\Omega}_{\kappa, n+1}}\|\nabla \xi_u^{n+1}\|_{L^2}^2\right] + C\tau h^2 \mathbb{E}\left[\mathds{1}_{\tilde{\Omega}_{\kappa, n+1}}\|u(t_{n+1})\|_{H^2}^2\right]\notag\\
&+C\tau^2 + C\tau h^2 \mathbb{E}\left[\mathds{1}_{\tilde{\Omega}_{\kappa, n+1}}\|d_tu(t_{n+1})\|_{H^2}^2\right]\notag\\
\le&\, C\tau\mathbb{E}\left[\mathds{1}_{\tilde{\Omega}_{\kappa, n+1}}\|\nabla \xi_u^{n+1}\|_{L^2}^2\right] + C\tau h^2 \sup\limits_{s\le T}\mathbb{E}\left[\mathds{1}_{\tilde{\Omega}_{\kappa, n+1}}\|u(s)\|_{H^2}^2\right]\notag\\
&\,+C\tau^2 + C\tau h^2,\notag
\end{align}
where the mean value theorem, the $H^1$ stability of the $L^2$ project, and Lemma \ref{lem_20201226_4} are used in the second last inequality, and Lemmas \ref{lem_20201226_1}, \ref{lem_20200630_1}, \ref{lem_20200717_1} are used in the last inequality. 
The second term in \eqref{eq20210605_6} is bounded by
\begin{align} \label{eq20210606_2}
&\mathbb{E}\left[\mathds{1}_{\tilde{\Omega}_{\kappa, n+1}}\int_{t_n}^{t_{n+1}}(\nabla e_u^{n+1}, \nabla d_t e_u^{n+1}) ds\right]\\
=& \, \frac12\mathbb{E}\left[\mathds{1}_{\tilde{\Omega}_{\kappa, n+1}}\|\nabla e_u^{n+1}\|_{L^2}^2 ds\right]-\frac12\mathbb{E}\left[\mathds{1}_{\tilde{\Omega}_{\kappa, n+1}}\|\nabla e_u^{n}\|_{L^2}^2 ds\right] \notag\\
&+   \frac12\mathbb{E}\left[\mathds{1}_{\tilde{\Omega}_{\kappa, n+1}}\|\nabla e_u^{n+1}-\nabla e_u^n\|_{L^2}^2 ds\right] \notag\\
\ge&\, \frac12\mathbb{E}\left[\mathds{1}_{\tilde{\Omega}_{\kappa, n+1}}\|\nabla e_u^{n+1}\|_{L^2}^2 ds\right]-\frac12\mathbb{E}\left[\mathds{1}_{\tilde{\Omega}_{\kappa, n}}\|\nabla e_u^{n}\|_{L^2}^2 ds\right] \notag\\
&+   \frac12\mathbb{E}\left[\mathds{1}_{\tilde{\Omega}_{\kappa, n+1}}\|\nabla e_u^{n+1}-\nabla e_u^n\|_{L^2}^2 ds\right] \notag.
\end{align}

The first term on the right-hand side of \eqref{eq20210605_2} is
\begin{align}\label{eq20210606_5}
&\mathbb{E}\left[\mathds{1}_{\tilde{\Omega}_{\kappa, n+1}}\int_{t_n}^{t_{n+1}}  ( f(u(s)) - f_h^{n+1}, \xi_v^{n+1}) ds\right]\\
= &\, \mathbb{E}\left[\mathds{1}_{\tilde{\Omega}_{\kappa, n+1}}\int_{t_n}^{t_{n+1}}  ( f(u(s)) - f(u(t_{n+1})), \xi_v^{n+1}) ds\right]\notag\\
& + \mathbb{E}\left[\mathds{1}_{\tilde{\Omega}_{\kappa, n+1}}\int_{t_n}^{t_{n+1}}  ( f(u(t_{n+1}) - f_h^{n+1}, \xi_v^{n+1}) ds\right]\notag.
\end{align}
By Lemma \ref{lem_20201226_5}, the first term on the right-hand side of \eqref{eq20210606_5} is bounded as
\begin{align}\label{eq20210606_6}
&\mathbb{E}\left[\mathds{1}_{\tilde{\Omega}_{\kappa, n+1}}\int_{t_n}^{t_{n+1}}  ( f(u(s)) - f(u(t_{n+1})), \xi_v^{n+1}) ds\right]\\
\le& \,C\tau\mathbb{E}\left[\mathds{1}_{\tilde{\Omega}_{\kappa, n+1}}\|\xi_v^{n+1}\|_{L^2}^2\right] + C\kappa^{q-1}\tau^3.\notag
\end{align}
By Lemma \ref{lem20190928_1}, the second term on the right-hand side of \eqref{eq20210606_5} is bounded as
\begin{align}\label{eq20210606_7}
&\mathbb{E}\left[\mathds{1}_{\tilde{\Omega}_{\kappa, n+1}}\int_{t_n}^{t_{n+1}}  ( f(u(t_{n+1}) - f_h^{n+1}, \xi_v^{n+1}) ds\right]\\
\le&\,C\tau\mathbb{E}\Bigl[\mathds{1}_{\tilde{\Omega}_{\kappa, n+1}}\sum\limits_{j=1}^q \left(\|u(t_{n+1})\|_{H^1}^{2(j-1)}+\|u_h^{n+1}\|_{H^1}^{2(j-1)}\right) \notag\\
&\times\|u(t_{n+1})-u_h^{n+1}\|_{H^1}^2\Bigr]+ \tau\E \left[\mathds{1}_{\tilde{\Omega}_{\kappa, n+1}}\|\xi_v^{n+1}\|_{L^2}^2 \right] \notag\\
\le &\,C\kappa^{q-1}\tau \mathbb{E}\left[\mathds{1}_{\tilde{\Omega}_{\kappa, n+1}}\|\xi_u^{n+1}\|_{L^2}^2 \right] + C\kappa^{q-1}\tau h^2 + \tau\E \left[\mathds{1}_{\tilde{\Omega}_{\kappa, n+1}}\|\xi_v^{n+1}\|_{L^2}^2 \right] \notag.
\end{align}

By It\^{o} isometry, the second term on the right-hand side of \eqref{eq20210605_2} is bounded as
\begin{align}\label{eq20210606_8}
&\mathbb{E}\left[\mathds{1}_{\tilde{\Omega}_{\kappa, n+1}}\int_{t_n}^{t_{n+1}} (g(u(s)) - g(u_h^n), \xi_v^{n+1}) dW(s)\right]\\
=& \mathbb{E}\left[\mathds{1}_{\tilde{\Omega}_{\kappa, n+1}}\int_{t_n}^{t_{n+1}} (g(u(s)) - g(u(t_n)), \xi_v^{n+1}-\xi_v^n) dW(s)\right]\notag\\
&+\mathbb{E}\left[\mathds{1}_{\tilde{\Omega}_{\kappa, n+1}}\int_{t_n}^{t_{n+1}} (g(u(t_n))-g(u_h^n), \xi_v^{n+1}-\xi_v^n) dW(s)\right]\notag\\
\le& \frac14\mathbb{E}\left[\mathds{1}_{\tilde{\Omega}_{\kappa, n+1}} \|\xi_v^{n+1}-\xi_v^n\|_{L^2}^2\right] + C\tau^3 + C\tau h^4\sup\limits_{s\le T}\mathbb{E}\left[\mathds{1}_{\tilde{\Omega}_{\kappa, n+1}}\|u(s)\|_{H^2}^2\right]\notag\\
& + C\tau\mathbb{E}\left[\mathds{1}_{\tilde{\Omega}_{\kappa, n+1}} \|\xi_u^{n}\|_{L^2}^2 \right] \notag.
\end{align}

Combining \eqref{eq20210605_1}-\eqref{eq20210606_8}, we obtain the estimate
\begin{align*}
& \frac12\mathbb{E}\left[\mathds{1}_{\tilde{\Omega}_{\kappa, n+1}}\|\xi_u^{n+1}\|_{L^2}^2\right] - \frac12\mathbb{E}\left[\mathds{1}_{\tilde{\Omega}_{\kappa, n}}\|\xi_u^n\|_{L^2}^2\right] + \frac12\mathbb{E}\left[\mathds{1}_{\tilde{\Omega}_{\kappa, n+1}}\|\xi_u^{n+1}-\xi_u^n\|_{L^2}^2\right] \\
&+\frac12\mathbb{E}\left[\mathds{1}_{\tilde{\Omega}_{\kappa, n+1}}\|\xi_v^{n+1}\|_{L^2}^2\right] - \frac12\mathbb{E}\left[\mathds{1}_{\tilde{\Omega}_{\kappa, n}}\|\xi_v^n\|_{L^2}^2\right] + \frac12\mathbb{E}\left[\mathds{1}_{\tilde{\Omega}_{\kappa, n+1}}\|\xi_v^{n+1}-\xi_v^n\|_{L^2}^2\right]\notag\\
&+\frac12\mathbb{E}\left[\mathds{1}_{\tilde{\Omega}_{\kappa, n+1}}\|\nabla e_u^{n+1}\|_{L^2}^2 ds\right]-\frac12\mathbb{E}\left[\mathds{1}_{\tilde{\Omega}_{\kappa, n}}\|\nabla e_u^{n}\|_{L^2}^2 ds\right] \notag\\
&+ \frac12\mathbb{E}\left[\mathds{1}_{\tilde{\Omega}_{\kappa, n+1}}\|\nabla e_u^{n+1}-\nabla e_u^n\|_{L^2}^2 ds\right]\notag\\
 \le &\,C\kappa^{q-1}\tau\mathbb{E}\left[\mathds{1}_{\tilde{\Omega}_{\kappa, n+1}} \|\xi_u^{n+1}\|_{L^2}^2 \right]+C\tau\mathbb{E}\left[\mathds{1}_{\tilde{\Omega}_{\kappa, n+1}} \|\xi_v^{n+1}\|_{L^2}^2 \right] + C\tau^2 + C\kappa^{q-1}\tau h^2\notag\\
 &+ C\tau h^2 +C\tau\mathbb{E}\left[\mathds{1}_{\tilde{\Omega}_{\kappa, n+1}}\|\nabla \xi_h^{n+1}\|_{L^2}^2\right]+ \frac14 \mathbb{E}\left[\mathds{1}_{\tilde{\Omega}_{\kappa, n+1}}\|\nabla e_u^{n+1}-\nabla e_u^n\|_{L^2}^2\right]\notag\\
 &+\frac14\mathbb{E}\left[\mathds{1}_{\tilde{\Omega}_{\kappa, n+1}} \|\xi_v^{n+1}-\xi_v^n\|_{L^2}^2\right]+ C\tau\mathbb{E}\left[\mathds{1}_{\tilde{\Omega}_{\kappa, n+1}} \|\xi_u^{n}\|_{L^2}^2 \right]+C\kappa^{q-1}\tau^3.\notag
\end{align*}
By choosing $\kappa^{q-1}=C\ln(h^{-\beta})$, where $\beta>0$ is small enough, taking the summation over $n$ from $0$ to $\ell-1$, and applying Gronwall's inequality, we obtain
\begin{align}\label{eq20210606_4}
& \mathbb{E}\left[\mathds{1}_{\tilde{\Omega}_{\kappa, \ell}}\|\xi_u^{\ell}\|_{L^2}^2\right]+\mathbb{E}\left[\mathds{1}_{\tilde{\Omega}_{\kappa, \ell}}\|\nabla e_u^{\ell}\|_{L^2}^2 ds\right]+\mathbb{E}\left[\mathds{1}_{\tilde{\Omega}_{\kappa, \ell}}\|\xi_v^{\ell}\|_{L^2}^2\right] \\
&+ \sum\limits_{n=0}^{\ell-1}\mathbb{E}\left[\mathds{1}_{\tilde{\Omega}_{\kappa, n+1}}\|\xi_v^{n+1}-\xi_v^n\|_{L^2}^2\right]+ \sum\limits_{n=0}^{\ell-1}\mathbb{E}\left[\mathds{1}_{\tilde{\Omega}_{\kappa, n+1}}\|\xi_u^{n+1}-\xi_u^n\|_{L^2}^2\right]\notag\\
&+ \sum\limits_{n=0}^{\ell-1}\mathbb{E}\left[\mathds{1}_{\tilde{\Omega}_{\kappa, n+1}}\|\nabla e_u^{n+1}-\nabla e_u^n\|_{L^2}^2 ds\right]\notag\\
\le &\,(C\kappa^{q-1}h^2 + C\tau+C\kappa^{q-1}\tau^2)h^{-\beta}\notag\\
\le &\,(Ch^2|\ln h| + C\tau+C\tau^2|\ln h|)h^{-\beta}\notag. 
\end{align}
Using the Markov's inequality, discrete Burkholder–Davis–Gundy inequalities \cite{beiglbock2015pathwise,burkholder1966martingale,burkholder1970extrapolation,davis1970intergrability}, equation \eqref{eq:stability-continuous}, and Theorem \ref{thm20180815_1}, we have the following property
\begin{align}\label{eq20160601_5}
\P\left[\tilde{\Omega}_{\kappa,\ell}\right]&\ge 1-\frac{\E\left[\max\limits_{1\le n\le\ell}\|u_h^n\|_{H^1}^2+\max\limits_{s\le t_\ell}\|u(s)\|_{H^1}^2\right]}{(C\ln(h^{-\beta}))^{\frac{1}{q-1}}}\\
&\ge 1- \frac{1}{\beta^{\frac{1}{q-1}}(-\ln h)^{\frac{1}{q-1}}}\rightarrow1\ \text{as}\ h\rightarrow0\notag.
\end{align}
By combining \eqref{eq20210606_4} with the properties of the $L^2$-projection, we finish the proof of this theorem.
\end{proof}

\begin{remark}
\begin{enumerate}
\item We focus on the nonlinear case when $q>1$ since it is much easier to analyze the linear case when $q=1$;
\item The original form \eqref{dfem} and the mixed form \eqref{eq20210623_1}--\eqref{eq20210623_2} are mathematically equivalent, but there exists some difficulties in analyzing the noise term that might not be easily circumvented if the original form is used.
\end{enumerate}
\end{remark}

\section{Numerical Tests}\label{sec-5}
In this section, we provide various numerical tests to validate our
theoretical results. We consider the stability and error estimates of
our proposed numerical schemes based on different nonlinear drift terms
$f(u)$ and different diffusion terms $g(u)$ in both one-dimensional
and two-dimensional cases. 
The regular Monte-Carlo method is used to
compute the stochastic term, and $5000$ sample points are used for all
the tests below.
\smallskip
\paragraph{\bf Test 1.} Consider the one-dimensional stochastic wave
equations \eqref{sac_s}--\eqref{eq20191102_1} with the initial
conditions
\begin{equation*}
h_1(x)=\cos(\pi x),\qquad h_2(x)=0,
\end{equation*}
and different nonlinear drift and diffusion terms outlined below.

First, we consider the nonlinear drift and diffusion terms chosen to
be $f(u)=-u-u^3$ and $g(u)=u$. We evaluate the following
errors, $\left\{\sup\limits_{0 \leq n \leq N} \E \bigl[ \| e^n
\|^2_{L^2} \bigr]\right\}^{\frac12}$ ($L^2$ error),
$\left\{\sup\limits_{0 \leq n \leq N} \E \bigl[ \|\nabla e^n
\|^2_{L^2} \bigr]\right\}^{\frac12}$ ($H^1$ error), and
$\left\{\sup\limits_{0 \leq n \leq N} \E \bigl[ \|d_te^n \|^2_{L^2}
\bigr]\right\}^{\frac12}$ ($d_tL^2$ error). Table \ref{tab_error1}
show these errors and their convergence rates with respect to space,
when the time step is fixed to be $\tau = 1\times 10^{-3}$, from which
the spatial order of 2 ($L^2$ error), 1 ($H^1$ error), 2
($d_tL^2$ error) can be observed. Similarly, we fix the spatial step
size and report the errors and their convergence rates with respect to
time in Table \ref{tab_error2}, where the temporal order of 1 are
observed for all three errors. 

\begin{table}[!htbp]
\centering 
\footnotesize
\begin{tabular}{|l||c|c||c|c||c|c|}
\hline  
&  $L^2$ error \quad & order &   
$H^1$ error \quad & order &   
$d_tL^2$ error\quad & order \\ \hline    
$h=1/4$ & $7.902\times10^{-2}$ & --- &   
$8.798\times10^{-1}$ & --- &   
$1.488\times10^{-1}$ & --- \\ \hline    
$h=1/8$ & $1.602\times10^{-2}$ & 2.302 &   
$4.178\times10^{-1}$ & 1.074 &   
$3.369\times10^{-2}$ & 2.143\\ \hline    
$h=1/16$ & $3.792\times10^{-3}$ & 2.079 &   
$2.063\times10^{-1}$ & 1.018&   
$8.079\times10^{-3}$ & 2.060\\ \hline    
$h=1/32$ & $9.349\times10^{-4}$  & 2.020 &   
$1.028\times10^{-1}$ & 1.004&   
$1.999\times10^{-3}$ &  2.015\\ \hline    
$h=1/64$ & $2.329\times10^{-4}$ & 2.005 &   
$5.138\times10^{-2}$ & 1.001 &   
$4.984\times10^{-4}$  & 2.004\\ \hline    
\end{tabular}
\caption{Test 1 (a): Spatial errors and convergence rates when $\tau =
1\times 10^{-3}$, $T = 0.01$.}
\label{tab_error1}
\end{table}

\begin{table}[!htbp]
\vspace*{-5mm}
\centering 
\footnotesize
\begin{tabular}{|l||c|c||c|c||c|c|}
\hline  
&  $L^2$ error \quad & order &   
$H^1$ error \quad & order &   
$d_tL^2$ error\quad & order \\ \hline    
$\tau=0.1$ & $3.506\times10^{-2}$ & --- &   
$1.102\times10^{-1}$ & --- &   
$2.630\times10^{-1}$ & --- \\ \hline    
$\tau=0.1/2$ & $1.495\times10^{-2}$ & 1.230 &   
$4.708\times10^{-2}$ & 1.227 &   
$1.577\times10^{-1}$ & 0.738\\ \hline    
$\tau=0.1/4$ & $6.379\times10^{-3}$ & 1.229 &   
$2.017\times10^{-2}$ & 1.223&   
$8.577\times10^{-2}$ & 0.879\\ \hline    
$\tau=0.1/8$ & $2.860\times10^{-3}$  & 1.157 &   
$9.063\times10^{-3}$ & 1.154&   
$4.451\times10^{-2}$ &  0.946\\ \hline    
$\tau=0.1/16$ & $1.341\times10^{-3}$ & 1.093 &   
$4.253\times10^{-3}$ & 1.092 &   
$2.278\times10^{-2}$  & 0.966\\ \hline  
$\tau=0.1/32$ & $6.572\times10^{-4}$ & 1.029 &   
$2.084\times10^{-3}$ & 1.029 &   
$1.164\times10^{-2}$  & 0.969\\ \hline    
\end{tabular}
\caption{Test 1 (a): Temporal errors and convergence rates when $h=1/128$, $T = 0.4 $.}
\label{tab_error2}
\end{table}

Next, the nonlinear drift and diffusion terms are chosen to be $f(u)=-u-u^{11}$ and $g(u)=u$, and the corresponding results are demonstrated in Tables \ref{tab_error3} and \ref{tab_error4}. The convergence rates are consistent with those of the previous test where $f(u)=-u-u^3$.

\begin{table}[!htbp]
\centering 
\footnotesize
\begin{tabular}{|l||c|c||c|c||c|c|}
\hline  
&  $L^2$ error \quad & order &   
$H^1$ error \quad & order &   
$d_tL^2$ error\quad & order \\ \hline    
$h=1/4$ & $7.982\times10^{-2}$ & --- &   
$8.828\times10^{-1}$ & --- &   
$1.938\times10^{-1}$ & --- \\ \hline    
$h=1/8$ & $1.614\times10^{-2}$ & 2.306 &   
$4.207\times10^{-1}$ & 1.069 &   
$3.445\times10^{-2}$ & 2.492\\ \hline    
$h=1/16$ & $3.801\times10^{-3}$ & 2.086 &   
$2.068\times10^{-1}$ & 1.024&   
$8.168\times10^{-3}$ & 2.076\\ \hline    
$h=1/32$ & $9.369\times10^{-4}$  & 2.020 &   
$1.031\times10^{-1}$ & 1.004&   
$2.017\times10^{-3}$ &  2.018\\ \hline    
$h=1/64$ & $2.334\times10^{-4}$ & 2.005 &   
$5.149\times10^{-2}$ & 1.002 &   
$5.022\times10^{-4}$  & 2.006\\ \hline    
\end{tabular}
\caption{Test 1 (b): Spatial errors and convergence rates when $\tau = 1\times 10^{-3}$, $T = 0.01$.}
\label{tab_error3}
\end{table}

\begin{table}[!htbp]
\vspace*{-5mm}
\centering 
\footnotesize
\begin{tabular}{|l||c|c||c|c||c|c|}
\hline  
&  $L^2$ error \quad & order &   
$H^1$ error \quad & order &   
$d_tL^2$ error\quad & order \\ \hline    
$\tau=0.1$ & $3.206\times10^{-2}$ & --- &   
$1.008\times10^{-1}$ & --- &   
$2.526\times10^{-1}$ & --- \\ \hline    
$\tau=0.1/2$ & $1.342\times10^{-2}$ & 1.256 &   
$4.242\times10^{-2}$ & 1.249 &   
$1.504\times10^{-1}$ & 0.748\\ \hline    
$\tau=0.1/4$ & $5.692\times10^{-3}$ & 1.237 &   
$1.830\times10^{-2}$ & 1.213&   
$8.125\times10^{-2}$ & 0.888\\ \hline    
$\tau=0.1/8$ & $2.536\times10^{-3}$  & 1.166 &   
$8.338\times10^{-3}$ & 1.134&   
$4.236\times10^{-2}$ &  0.940\\ \hline    
$\tau=0.1/16$ & $1.196\times10^{-3}$ & 1.084 &   
$3.996\times10^{-3}$ & 1.061 &   
$2.164\times10^{-2}$  & 0.969\\ \hline  
$\tau=0.1/32$ & $5.938\times10^{-4}$ & 1.010 &   
$1.995\times10^{-3}$ & 1.002 &   
$1.106\times10^{-2}$  & 0.968\\ \hline    
\end{tabular}
\caption{Test 1 (b): Temporal errors and convergence rates when $h=1/128$, $T = 0.4 $.}
\label{tab_error4}
\end{table}

Lastly, the nonlinear drift and diffusion terms are chosen to be $f(u)=-u-u^{3}$ and $g(u)=\sqrt{u^2+0.01}$, and the corresponding results are included in Tables \ref{tab_error5} and  \ref{tab_error6}. We observe that the convergence rates of the $L^2$ and the $H^1$ errors are consistent with the above cases, but the temporal order of the $d_tL^2$ error loses a half.

\begin{table}[!htbp]
\centering 
\footnotesize
\begin{tabular}{|l||c|c||c|c||c|c|}
\hline  
&  $L^2$ error \quad & order &   
$H^1$ error \quad & order &   
$d_tL^2$ error\quad & order \\ \hline    
$h=1/4$ & $7.901\times10^{-2}$ & --- &   
$8.798\times10^{-1}$ & --- &   
$1.469\times10^{-1}$ & --- \\ \hline    
$h=1/8$ & $1.602\times10^{-2}$ & 2.302 &   
$4.178\times10^{-1}$ & 1.074 &   
$3.364\times10^{-2}$ & 2.127\\ \hline    
$h=1/16$ & $3.792\times10^{-3}$ & 2.079 &   
$2.063\times10^{-1}$ & 1.018&   
$8.053\times10^{-3}$ & 2.063\\ \hline    
$h=1/32$ & $9.349\times10^{-4}$  & 2.020 &   
$1.028\times10^{-1}$ & 1.005&   
$1.991\times10^{-3}$ &  2.016\\ \hline    
$h=1/64$ & $2.329\times10^{-4}$ & 2.005 &   
$5.138\times10^{-2}$ & 1.001 &   
$4.964\times10^{-4}$  & 2.004\\ \hline    
\end{tabular}
\caption{Test 1 (c): Spatial errors and convergence rates when $\tau = 1\times 10^{-3}$, $T = 0.01$.}
\label{tab_error5}
\end{table}

\begin{table}[!htbp]
\vspace*{-5mm}
\centering 
\footnotesize
\begin{tabular}{|l||c|c||c|c||c|c|}
\hline  
&  $L^2$ error \quad & order &   
$H^1$ error \quad & order &   
$d_tL^2$ error\quad & order \\ \hline    
$\tau=0.1$ & $1.132\times10^{-1}$ & --- &   
$3.502\times10^{-1}$ & --- &   
$1.052\times10^{-1}$ & --- \\ \hline    
$\tau=0.1/2$ & $7.337\times10^{-2}$ & 0.626 &   
$2.292\times10^{-1}$ & 0.612 &   
$7.203\times10^{-2}$ & 0.546\\ \hline    
$\tau=0.1/4$ & $4.198\times10^{-2}$ & 0.805 &   
$1.318\times10^{-1}$ & 0.798&   
$4.773\times10^{-2}$ & 0.594\\ \hline    
$\tau=0.1/8$ & $2.242\times10^{-2}$  & 0.905 &   
$7.068\times10^{-2}$ & 0.899&   
$3.184\times10^{-2}$ &  0.584\\ \hline    
$\tau=0.1/16$ & $1.158\times10^{-2}$ & 0.953 &   
$3.666\times10^{-2}$ & 0.947 &   
$2.169\times10^{-2}$  & 0.554\\ \hline  
$\tau=0.1/32$ & $5.890\times10^{-3}$ & 0.975 &   
$1.871\times10^{-2}$ & 0.970 &   
$1.495\times10^{-2}$  & 0.537\\ \hline    
\end{tabular}
\caption{Test 1 (c): Temporal errors and convergence rates when $h=1/128$, $T = 0.4 $.}
\label{tab_error6}
\end{table}

\smallskip
\noindent {\bf Test 2.} Consider the one-dimensional stochastic wave equations \eqref{sac_s}--\eqref{eq20191102_1} with the following initial conditions (with less regularity)
\begin{equation*}
h_1(x)=0,\qquad h_2(x)=\max\{0,1-4|x-0.5|\}.
\end{equation*}
The nonlinear drift and diffusion terms are chosen to be $f(u)=-u-u^3$ and $g(u)=u$. 
Table \ref{tab_error7} and Table \ref{tab_error8} show the $L^2$ error, the $H^1$ error, the $d_tL^2$ error, and their convergence rates in both space and time. The spatial convergence rates are the same as in the previous cases, and the temporal convergence rates of the $L^2$, $H^1$, and $d_tL^2$ errors are approximately $1.0$, $1.0$, and $0.75$.

\begin{table}[!htbp]
\centering 
\footnotesize
\begin{tabular}{|l||c|c||c|c||c|c|}
\hline  
&  $L^2$ error \quad & order &   
$H^1$ error \quad & order &   
$d_tL^2$ error\quad & order \\ \hline    
$h=1/16$ & $2.565\times10^{-4}$ & --- &   
$2.335\times10^{-2}$ & --- &   
$4.586\times10^{-3}$ & ---\\ \hline    
$h=1/32$ & $6.355\times10^{-5}$ & 2.013 &   
$1.097\times10^{-2}$ & 1.089&   
$1.896\times10^{-3}$ & 1.275\\ \hline    
$h=1/64$ & $1.642\times10^{-5}$  & 1.952 &   
$5.500\times10^{-3}$ & 0.996&   
$5.391\times10^{-4}$ &  1.814\\ \hline    
$h=1/128$ & $4.142\times10^{-6}$ & 1.987 &   
$2.746\times10^{-3}$ & 1.002 &   
$1.410\times10^{-4}$  & 1.935\\ \hline    
$h=1/256$ & $1.038\times10^{-6}$ & 1.997 &   
$1.373\times10^{-3}$ & 1.001 &   
$3.565\times10^{-5}$  & 1.984\\ \hline    
\end{tabular}
\caption{Test 2: Spatial errors and convergence rates when $\tau = 5\times 10^{-3}$, $T = 5\times 10^{-2}$.}
\label{tab_error7}
\end{table}

\begin{table}[!htbp]
\vspace*{-5mm}
\centering 
\footnotesize
\begin{tabular}{|l||c|c||c|c||c|c|}
\hline  
&  $L^2$ error \quad & order &   
$H^1$ error \quad & order &   
$d_tL^2$ error\quad & order \\ \hline    
$\tau=0.1/8$ & $1.324\times10^{-3}$ & --- &   
$7.489\times10^{-3}$ & --- &   
$4.653\times10^{-2}$ & ---\\ \hline    
$\tau=0.1/16$ & $6.437\times10^{-4}$ & 1.040 &   
$3.963\times10^{-3}$ & 0.918&   
$3.130\times10^{-2}$ & 0.572\\ \hline    
$\tau=0.1/32$ & $3.163\times10^{-4}$  & 1.025 &   
$1.935\times10^{-3}$ & 1.035&   
$1.920\times10^{-2}$ &  0.705\\ \hline    
$\tau=0.1/64$ & $1.549\times10^{-4}$ & 1.030 &   
$9.583\times10^{-4}$ & 1.014 &   
$1.152\times10^{-2}$  & 0.738\\ \hline  
$\tau=0.1/128$ & $7.659\times10^{-5}$ & 1.016 &   
$4.974\times10^{-4}$ & 0.946 &   
$6.908\times10^{-3}$  & 0.737\\ \hline    
\end{tabular}
\caption{Test 2: Temporal errors and convergence rates when $h=1/512$, $T = 1$.}
\label{tab_error8}
\end{table}

\smallskip
\noindent {\bf Test 3.} Consider the two-dimensional stochastic wave equations \eqref{sac_s}--\eqref{eq20191102_1} with the following initial conditions
\begin{equation*}
h_1(x,y)=\cos(\pi x)\cos(2\pi y),\qquad h_2(x,y)=0.
\end{equation*}
We investigate the stability (in different norms) of the proposed methods with various nonlinear drift and diffusion terms. For comparison, we also include the results of the deterministic equations.  

Figures \ref{fig2}, \ref{fig3}, and \ref{fig4} provide the time history of the stability in $L^2$, $H^1$, and $d_tL^2$ norms of the stochastic (left) and deterministic (right) solutions. The transparent shaded regions in the left figures are possible trajectories of all sample points, and the solid lines represent the average of all trajectories. In Figure \ref{fig2}, the nonlinear drift and diffusion terms are chosen to be $f(u)=-u-u^3$ and $g(u)=u$. In Figure \ref{fig3}, the nonlinear drift and diffusion terms are chosen to be $f(u)=-u-u^7$ and $g(u)=u$. In Figure \ref{fig4}, the nonlinear drift and diffusion terms are chosen to be $f(u)=-u-u^3$ and $g(u)=\sqrt{u^2+1}$. From these figures, we can observe a good match between the deterministic and stochastic solutions, and the stochastic numerical solutions do not blow up, which is consistent with the theoretical results provided in this paper. 

\begin{figure}[!htbp]
\centering 
\includegraphics[width=0.49\textwidth]{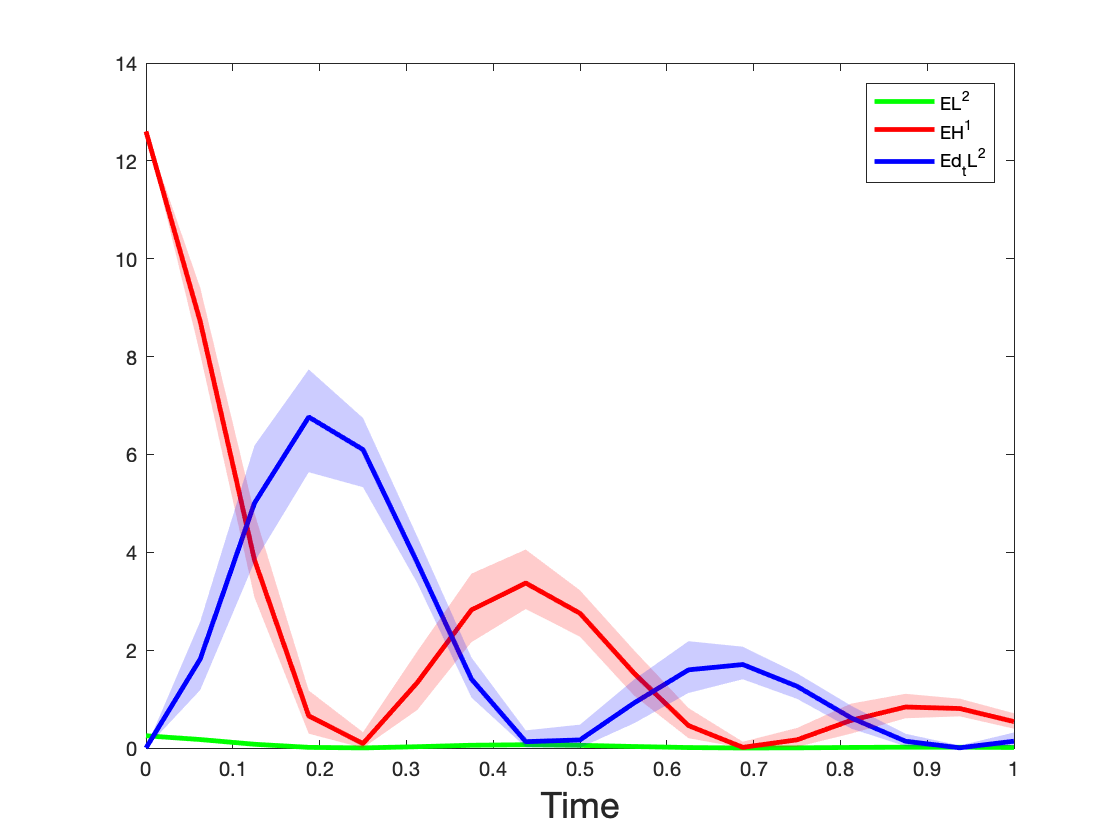} 
\includegraphics[width=0.49\textwidth]{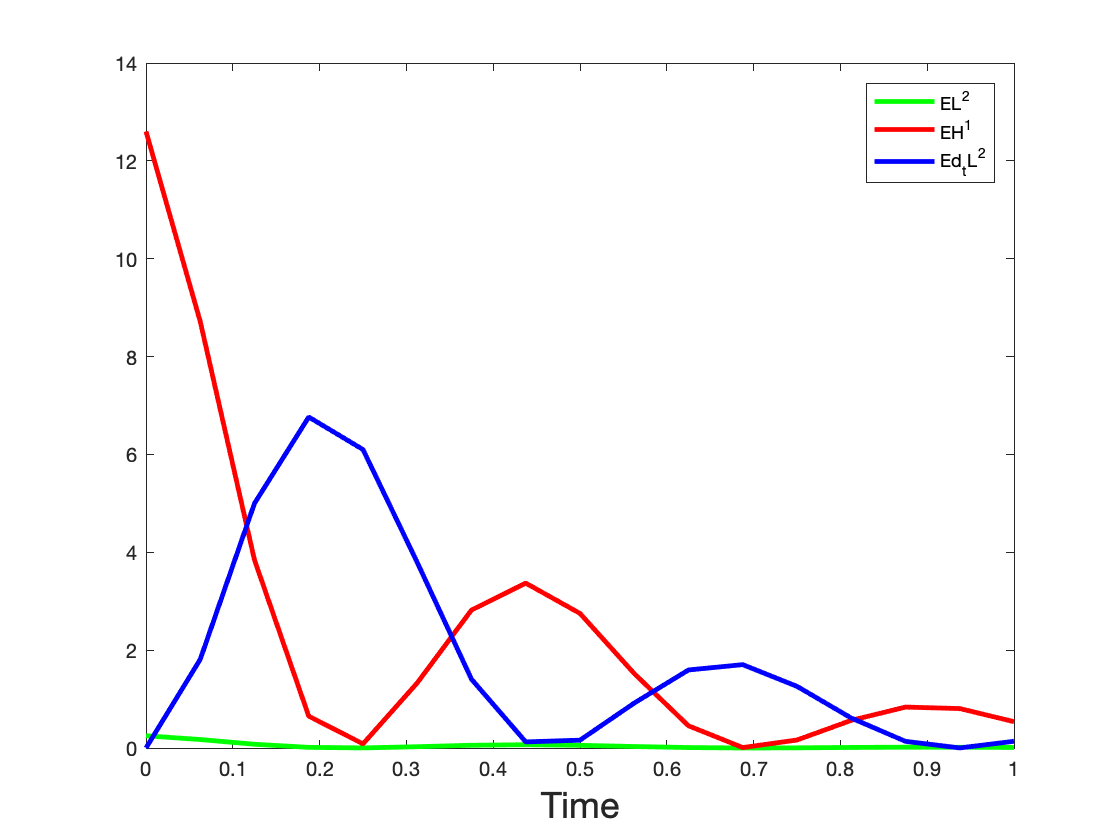} 
\caption{Test 3: The stability in the stochastic case (left), and the stability in the deterministic case (right). Here $f(u)=-u-u^3$ and $g(u)=u$.}
\label{fig2}
\end{figure}

\begin{figure}[!htbp]
\centering 
\includegraphics[width=0.49\textwidth]{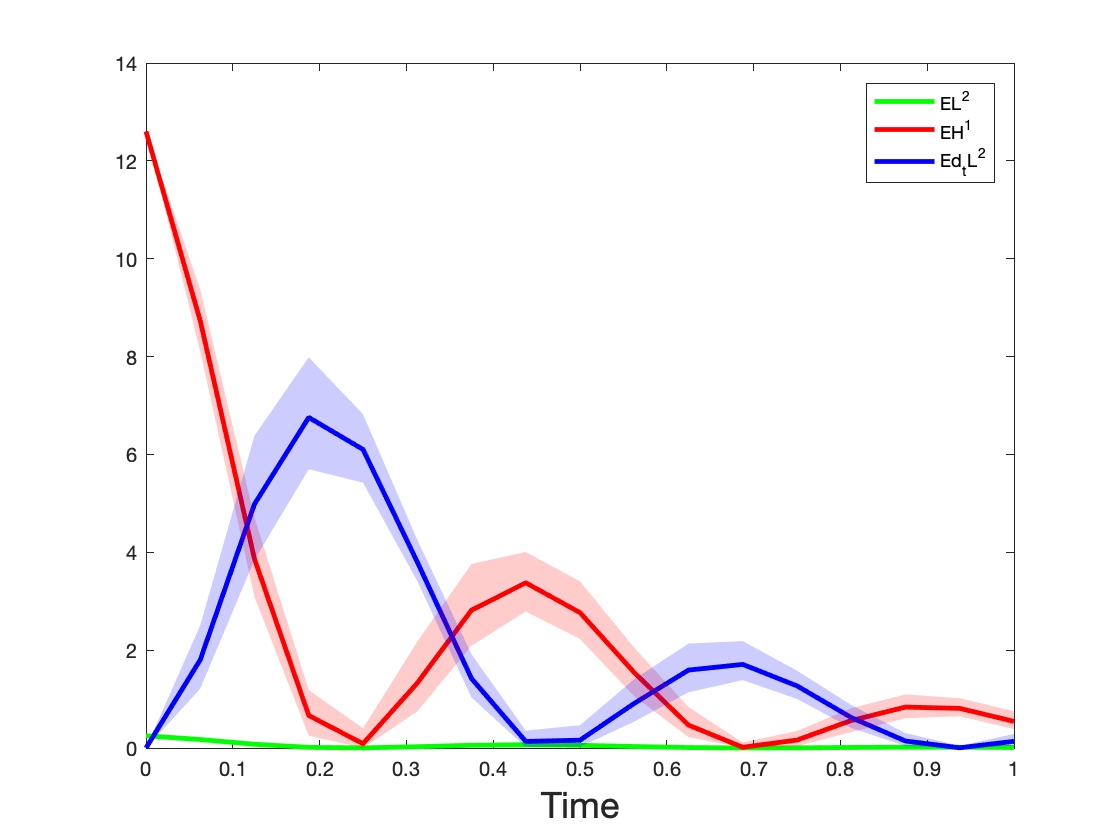} 
\includegraphics[width=0.49\textwidth]{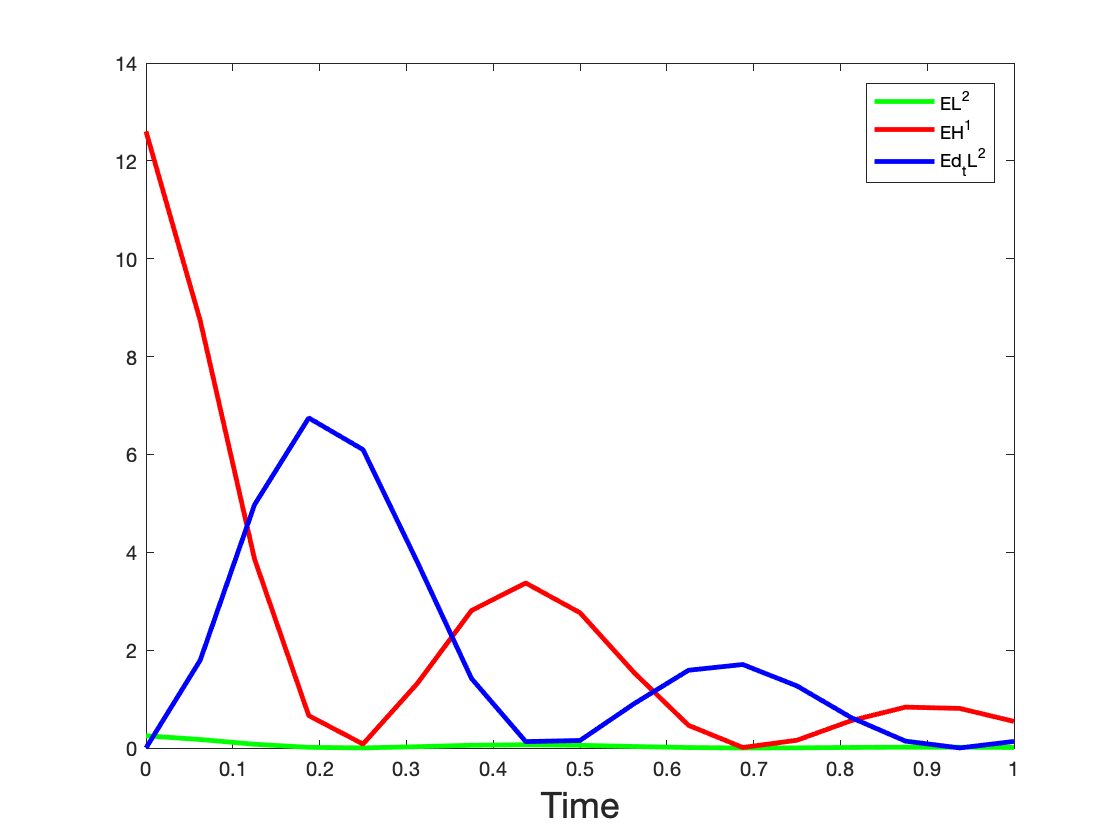} 
\caption{Test 3: The stability in the stochastic case (left), and the stability in the deterministic case (right). Here $f(u)=-u-u^7$ and $g(u)=u$.}
\label{fig3}
\end{figure}

\begin{figure}[!htbp]
\centering 
\includegraphics[width=0.49\textwidth]{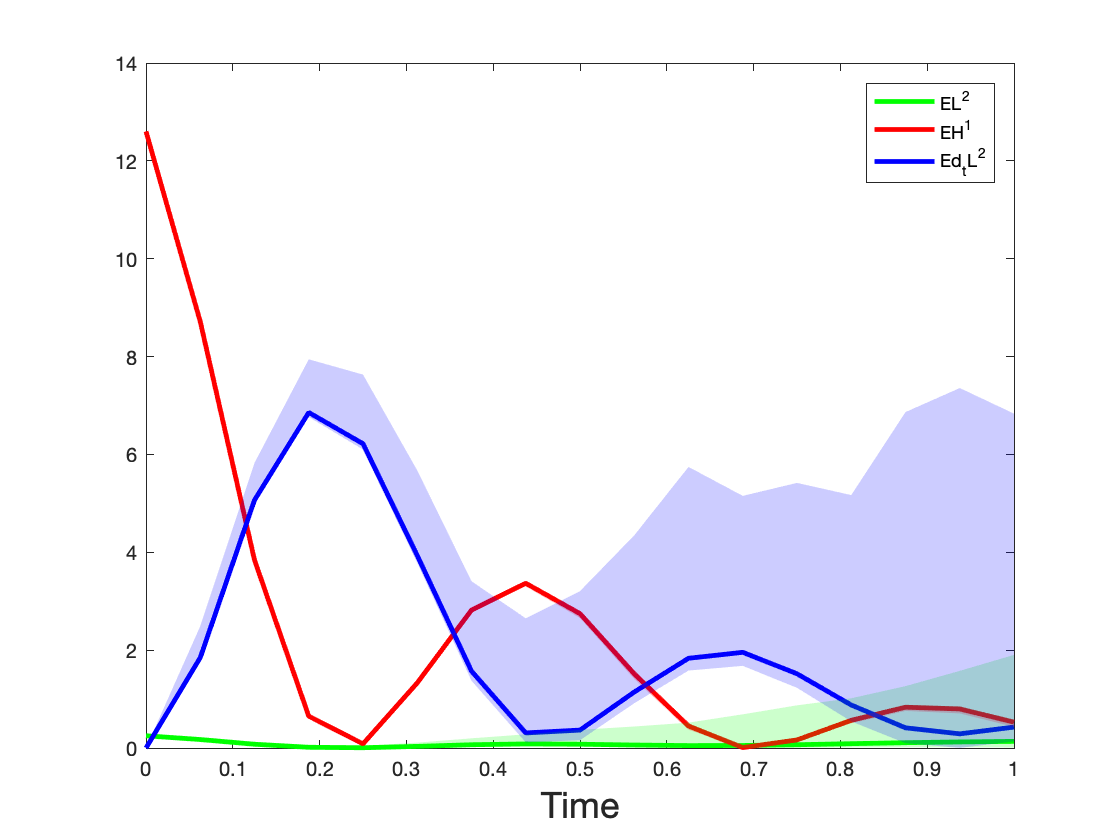} 
\includegraphics[width=0.49\textwidth]{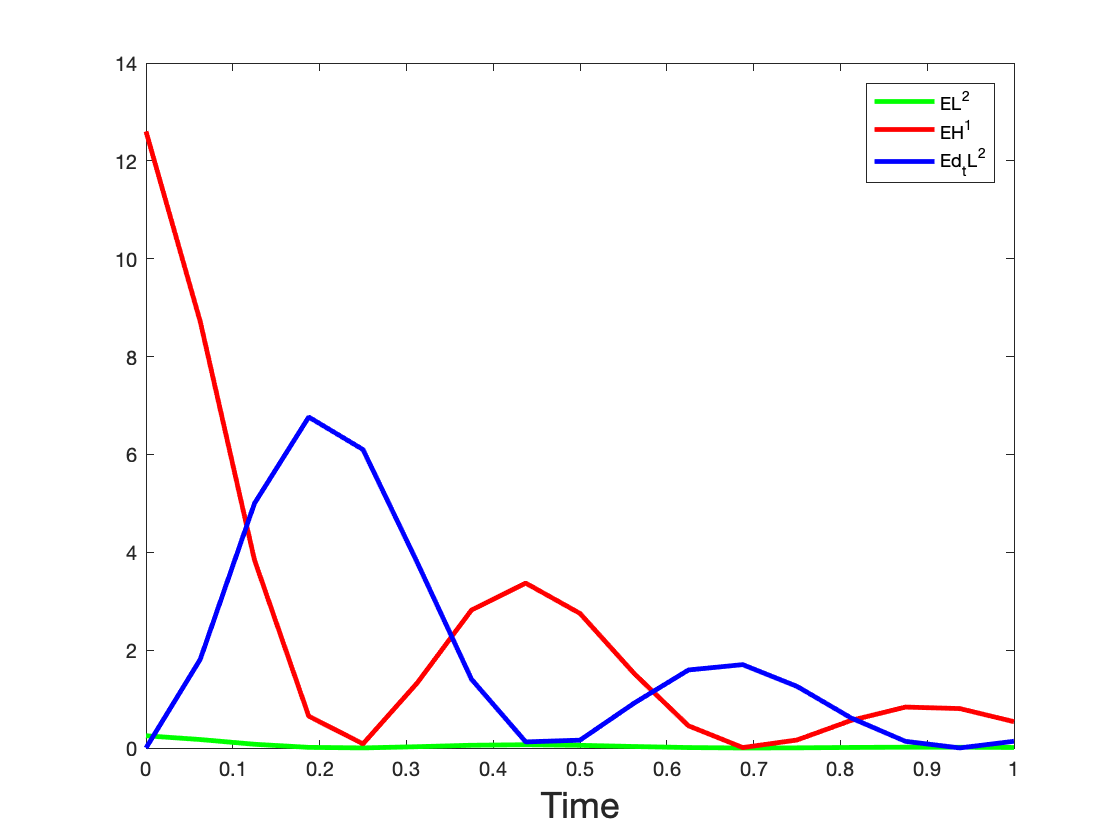} 
\caption{Test 3: The stability in the stochastic case (left), and the stability in the deterministic case (right). Here $f(u)=-u-u^3$ and $g(u)=\sqrt{u^2+1}$.}
\label{fig4}
\end{figure}


\end{document}